\documentclass{amsart}

\usepackage[all]{xy}
\usepackage{pictexwd, epsfig, amsthm, amssymb, amsmath, graphicx, psfrag,amsxtra, latexsym}
\usepackage{enumerate}

\newcommand{\mQ}{\mathbb Q}
\newcommand{\mZ}{\mathbb Z}
\newcommand{\mR}{\mathbb R}
\newcommand{\mF}{\mathbb F}
\newcommand{\mH}{\mathbb H}
\newcommand{\mV}{\mathcal V}
\newcommand{\mG}{\mathcal G}
\newcommand{\tK}{\tilde K}

\newcommand{\g}{\gamma}

\DeclareMathOperator{\fix}{Fix} 
\DeclareMathOperator{\stab}{Stab}  \DeclareMathOperator{\G}{\Gamma}

\DeclareMathOperator{\at}{Aut} 
\DeclareMathOperator{\iso}{Isom} 
 \DeclareMathOperator{\aut}{Aut}

\DeclareMathOperator{\sta}{Stab}  \DeclareMathOperator{\kz}{\mathbb K \mathbb
Z^{-\infty}} \DeclareMathOperator{\vc}{\mathcal V\mathcal C}

\DeclareMathOperator{\fin}{\mathcal F\mathcal I\mathcal N}

\theoremstyle{plain}
\newtheorem{Thm}{Theorem}

\newtheorem{corollary}[Thm]{Corollary}

\newtheorem{example}[Thm]{Example}

\theoremstyle{definition}

\theoremstyle{remark}

\begin{document}

\title[Lower algebraic $K$-theory of certain reflection groups.]
      {Lower algebraic $K$-theory of certain reflection groups.}
\author{Jean-Fran\c{c}ois\ Lafont}
\address{Department of Mathematics\\
         Ohio State University\\
         Columbus, OH  43210}
\email[Jean-Fran\c{c}ois\ Lafont]{jlafont@math.ohio-state.edu}
\author{Bruce A. Magurn}
\address{Department of Mathematics and Statistics\\
         Miami University\\
         Oxford, OH 45056}
\email[Bruce A. Magurn]{magurnba@muohio.edu}
\author{Ivonne J.\ Ortiz}
\address{Department of Mathematics and Statistics\\
         Miami University\\
         Oxford, OH 45056}
\email[Ivonne J.\ Ortiz]{ortizi@muohio.edu}

\begin{abstract}
For $P\subset \mathbb H^3$ a finite volume geodesic polyhedron, with
the property that all interior angles between incident faces are of
the form $\pi/m_{ij}$ ($m_{ij}\geq 2$ an integer), there is a
naturally associated Coxeter group $\Gamma _P$.  Furthermore, this
Coxeter group is a lattice inside the semi-simple Lie group
$O^+(3,1)=\iso(\mathbb H^3)$, with fundamental domain the original
polyhedron $P$. In this paper, we provide a procedure for computing the
lower algebraic $K$-theory of the integral group ring of such groups
$\Gamma _P$ in terms of the geometry of the polyhedron $P$. 
As an ingredient in the computation, we explicitly calculate the
$K_{-1}$ and $Wh$ of the groups $D_n$ and $D_n\times \mZ_2$, 
and we also summarize what is known about the $\tilde K_0$. 

\end{abstract}

\maketitle

%%%%%%%%%%%%%%%%%%%%%%%%%%%%%%%%%%%%%%%%%%%%%%%%%%%
\section{Introduction}

Algebraic $K$-theory is a family of covariant functors from the
category of rings to the category of abelian groups, and when
applied to a ring $R$, yields information about the category of
projective modules over $R$.  The algebraic $K$-theory functors (and
their relatives) are of great interest to topologists, particularly
when applied to integral group rings of discrete groups.  Indeed, it
was discovered in the 1960's that, for various natural problems in
geometric topology, obstructions to the solution (in high dimensions)
appeared as all the members of suitable $K$-theory groups.
%, and that conversely every
%obstruction could be realized as an inequivalent, unsolvable problem.

In view of this, one can understand the desire to obtain {\it
explicit} computations for the $K$-theory of integral group rings of
finitely generated groups. Recent work of Lafont-Ortiz \cite{LO1},
\cite{LO2} gave explicit computations for the lower algebraic
$K$-theory of all the {\it hyperbolic 3-simplex reflection groups}.
These are the lattices in the Lie group
$O^+(3,1)=\iso(\mathbb H^3)$, that are generated by reflections in the sides of a suitable
geodesic 3-simplex in $\mathbb H^3$.  Such groups were classified
by several different authors,
and there are precisely 32 of them up to isomorphism (see the 
discussion in the introduction of \cite{JKRT}).

Now consider $P\subset \mathbb H^3$ a finite volume geodesic
polyhedron, with the property that all the interior angles between
incident faces are of the form $\pi/m_{ij}$ ($m_{ij}\geq 2$ an
integer).  One can extend each of the (finitely many) faces to a
hyperplane (i.e. totally geodesic $\mH ^2$ embedded in $\mH ^3$),
and form the subgroup $\G _P \leq O^+(3,1)$ generated by reflections
in these hyperplanes. This group will always be a lattice in
$O^+(3,1)$, with fundamental domain the original polyhedron $P$.  We
will call such a group a {\it 3-dimensional hyperbolic reflection
group}.  A special case of this occurs if $P$ is a tetrahedron, in
which case the group $\G _P$ is one of the 32 hyperbolic
3-simplex reflection groups. This case of the tetrahedron is somewhat special, as
for most other combinatorial types of  {\it fixed} simple polytopes, there are in fact
infinitely many distinct groups $\G _P$ with polyhedron $P$ of the given combinatorial
type (see for instance the Appendix for the case of the cube). 
In this paper, our goal is to explain how the methods of \cite{LO2} can be
extended to provide computations of the lower
algebraic $K$-theory of the lattice $\G _P \leq O^+(3,1)$ for arbitrary polyhedron
$P$. 

We present background material in Section 2. In particular,
we remind the reader of existing results, that allow us to express
the lower algebraic $K$-theory of the group $\G _P$ (namely $Wh(\G_P)$ for $*=1$, $\tilde{K}_0(\mathbb Z \G_P)$ for $*=0$, and $K_*(\mathbb Z \G_P)$ for $* <0$) as a direct sum:
$$H_*^{\G _P}(E_{\fin}(\G _P);\mathbb K\mathbb Z^{-\infty})
\oplus \bigoplus _{V\in \mathcal V}H_*^{V}(E_{\fin}(V)\rightarrow*).
$$allowing us to break down its computation into that of the various summands.

The first term appearing
in the above splitting is a suitable equivariant generalized homology 
group of a certain space. There is a spectral sequence which allows
one to compute this homology group. The $E^2$-terms of this spectral
sequence are obtained by taking the homology of a certain chain 
complex. The groups appearing in this chain complex are given
by the lower algebraic $K$-theory of various finite groups, primarily 
dihedral groups $D_n$ and products $D_n\times \mZ_2$.
In Section 3, we provide explicit number theoretic formulas for 
the $K_{-1}$ and $Wh$ of these finite groups (Sections 3.1 and 3.4), and we summarize
what is known about the $\tilde K_0$ of these groups (Section 3.3); the $K_*$ for $* < -1$ are well known to vanish.

In Section 4, we analyze the chain complex, compute the 
$E^2$-terms in the spectral sequence, and use this information
to identify the term $H_*^{\G_P}(E_{\fin}(\G _P);\mathbb K\mathbb Z^{-\infty})$
within the range $*\leq 1$. 
For $*=0,1$, it is known that the remaining terms in the splitting
are  torsion. Specializing to the case $*=1$, this allows us (Theorem 4)
to give an explicit formula for the 
rationalized Whitehead group $Wh(\G _P) \otimes \mQ$, in terms of 
the combinatorics and geometry of the polyhedron $P$. Similarly, 
for $*\leq -1$, the remaining 
terms in the splitting are known to vanish. This allows us (Theorem 5)
to also provide similarly explicit, though much more complicated,
expressions for $K_{-1}(\mZ \G _P)$.

In Section 5, we focus on identifying the remaining terms in the splitting. 
Using geometric techniques, we show that only finitely many of these terms 
are non-zero. Furthermore, we show that the non-vanishing terms consist of
Bass Nil-groups, associated to various dihedral groups that can be 
identified from the geometry of the polyhedron $P$. 

Overall, the results encompassed in this paper give a general 
procedure for computing the 
lower algebraic $K$-theory of any such $\G _P$. More precisely, we
can use the geometry of the polyhedron $P$ to:
\begin{itemize}
\item give a completely explicit computation for $K_{-1}(\mZ \G _P)$,
\item give an expression for $\tilde K_0(\mZ \G _P)$ in terms of the 
$\tilde K_0$ of various dihedral groups, and products of dihedral
groups with $\mZ_2$ (the computation of which is a classical problem, see
Section 4.2), as well as certain Bass Nil-groups $NK_0(\mZ D_n)$,
\item give an expression for $Wh(\G _P)$ in terms of the Bass Nil-groups
$NK_1(\mZ D_n)$ associated to various dihedral groups (the calculation
of which is also a well known, difficult problem).
\end{itemize}
In particular, we see that the lower algebraic $K$-theory of $\G_P$
can be directly determined from the geometry of the polyhedron
$P$. 

Finally, in the Appendix to the paper (Section 6), we illustrate this process by 
computing the lower algebraic $K$-theory for a family of 
Coxeter groups  $\G _P$ whose associated polyhedron $P$ is the product
of an $n$-gon with an interval (see Example 7). As a second class of
examples, we compute the lower algebraic $K$-theory for an 
infinite family of Coxeter groups whose associated polyhedrons are
combinatorial cubes (see Example 8).

\vskip 20pt
%%%%%%%%%%%%%%%%%%%%%%%%%%%%%%%%%%%%%%%%%%%%%%%%%%%
\centerline{\bf Acknowledgments.}

\vskip 5pt

The work of the first author and the third author was partially supported by the NSF, 
under the grants DMS - 0606002 and DMS - 0805605. The first author was partly
supported by an Alfred P. Sloan research fellowship.

\vskip 20pt

%%%%%%%%%%%%%%%%%%%%%%%%%%%%%%%%%%%%%%%%%%%%%%%%%%%
\section{Background material: geodesic polyhedra in $\mH ^3$, Coxeter groups, and the FJIC}

In this section, we introduce some background material. We briefly discuss the
groups of interest (Section 2.1), and the possible polyhedra $P$ 
(Section 2.2). We also discuss how, given a Coxeter group $\G$, we can
algorithmically decide whether $\G \cong \G_P$ for some polyhedron $P\subset \mH^3$
(Section 2.3). Finally, we provide a review of the Farrell-Jones isomorphism
conjecture, and discuss its relevance to our problem (Section 2.4).

\subsection{Hyperbolic reflection groups}  Consider a geodesic
polyhedron $P\subset \mH^3$, having the property that every pair of
incident faces intersects at an (internal) angle $\pi/m_{ij}$, where
$m_{ij} \geq 2$ is an integer.  Associated to such a polyhedron, one
can form a labeled complete graph $\mG$ as follows:
\begin{itemize}
\item associate a vertex $v_i$ to every face $F_i$ of $P$,
\item if a pair of faces $F_i,F_j$ of $P$ intersect at an angle of $\pi/m_{ij}$, label
the corresponding edge of $\mG$ by the integer $m_{ij}$,
\item if a pair of faces of $P$ do not intersect, label the
corresponding edge of $\mG$ by $\infty$.
\end{itemize}
From the resulting labeled graph $\mG$, one can form a Coxeter group
$\G _P$ in the usual manner: one assigns a generator $x_i$ of order
two to each vertex $v_i$ of $\mathcal G$, and adds in relations
$(x_ix_j)^{m_{ij}}=1$ to every labeled edge of $\mG$ (with the
understanding that a relation is vacuous if the exponent is
$\infty$).

From the constraint on the angles of the polyhedron $P$, it is clear
that one has a homomorphism $\G _P \rightarrow O^+(3,1)$, obtained
by assigning to each generator $x_i\in \G _P$ the reflection in the
hyperplane extending the corresponding face $F_i$ of $P$.  This
morphism is in fact an embedding of $\G _P\hookrightarrow O^+(3,1)$
as a lattice, with fundamental domain precisely the original
polyhedron $P$.

Conversely, we observe that, from the Coxeter graph $\mG$, we can
readily recover both the combinatorial polyhedron $P$ and
the internal angles $\pi/m_{ij}$ associated to each edge in the
1-skeleton of $P$.

%%%%%%%%%%%%%%%%%%%%%%%%%%%%%%%%%%%%%%%%%%%%%%%%%%%

\subsection{Geodesic polyhedra in $\mH ^3$}  A natural question
is the realization question: given a combinatorial polyhedron $P$
with prescribed internal angles of the form $\pi/m_{e}$ ($m_{e}\geq 2$)
assigned to each edge $e$ of $P$, does it arise as a geodesic
polyhedron inside $\mH ^3$?  If so, we will say that the labeled
combinatorial polyhedron is {\em realizable} in $\mH^3$.

In fact, a celebrated result of Andreev \cite{An} provides a
complete characterization of finite volume geodesic polyhedra in
$\mH ^3$ with non-obtuse internal angles.  More precisely, given an
abstract combinatorial polyhedron $P$, with a $0<\theta_e\leq \pi/2$
assigned to each edge $e$, Andreev demonstrated that the following
two statements are equivalent:
\begin{enumerate}
\item the labeled combinatorial polyhedron $P$ is realizable in $\mH^3$,
and 
\item the collection $\theta_e$ satisfy a finite collection of
linear inequalities, which are explicitly given in terms of the
combinatorics of the polyhedron $P$.
\end{enumerate}

For a more specific discussion, we refer the reader to the recent
paper of Roeder, Hubbard, and Dunbar \cite{RHD}, or to the book of M. Davis
\cite[Section 6.10]{Da}.  The point we want
to emphasize is that, from a combinatorial polyhedron with
prescribed dihedral angles $\pi/m_e$, one can use Andreev's theorem
to easily check whether the labeled combinatorial polyhedron is
realizable in $\mH^3$.

%%%%%%%%%%%%%%%%%%%%%%%%%%%%%%%%%%%%%%%%%%%%%%%%%%%

\subsection{Coxeter groups as hyperbolic reflection groups}  The
Coxeter groups of interest here are canonically associated to a
certain class of geodesic polyhedra $P$ in $\mH^3$.  For these
Coxeter groups, our goal is to provide recipes for computing the
lower algebraic $K$-theory.  The reader might naturally be
interested in knowing, given a Coxeter group, whether it is one of
these 3-dimensional hyperbolic reflection groups.  We summarize 
here the procedure
for answering this question.  Let us assume that we are given a
Coxeter group $\G$ in terms of its complete Coxeter graph $\mG$
(i.e. the complete graph on the generators, with each edge labelled
by either an integer $\geq 2$, or by $\infty$).

First of all, we note that if we exclude the edges labelled $\infty$
in the graph $\mG$, we obtain a labeled graph $\mG ^\prime$.  If the
Coxeter group $\G$ was a hyperbolic reflection group, then $\mG
^\prime$ would have to be the dual graph to the corresponding
polyhedron $P$, and in particular, would have to be a {\it planar
graph}. Furthermore, the fact that $\mG^\prime$ is dual to a
polyhedron implies that it is 3-connected.  So from now on, let us
assume $\mG ^\prime$ is a planar, 3-connected graph.

Note that by a famous result of Steinitz \cite{St}, 3-connected, planar
graphs, are precisely the class of graphs that arise as 1-skeletons
of polyhedra (in this case, the dual of $P$). Furthermore, a
well-known result of Whitney \cite{Wh} states that 3-connected planar graphs
have a unique embedding in $S^2$. A detailed discussion of both these
results can be found in Ziegler's classic text, see \cite[Ch. 4]{Z}.
These two results now allow us to
deal with $\mG^\prime$ as a polyhedron, since it arises as the
1-skeleton of a unique polyhedron.

Secondly, let us assume that $P$ is one of our geodesic polyhedra,
and $\G _P$ is the associated Coxeter group.  If we take a vertex
$v\in P$, then the stabilizer of $v$ under the $\G _P$ has to be a
2-dimensional 
spherical reflection group. Since all the spherical Coxeter groups
are in fact spherical triangle groups, this immediately implies that at most
three faces of $P$ can contain $v$.  So in particular, if
$\mG^\prime$ is dual to the polyhedron $P$, we see that all the
faces of $\mG^\prime$ corresponding to non-ideal vertices of $P$
must in fact be triangles.

Thirdly, if we take an {\it ideal vertex} $v$ of $P$, then the
stabilizer of $v$ in $\G _P$ must be a planar Coxeter group, generated
by reflections in a planar polygon with all angles of the form 
$\pi/m_e$. Note that there are four such polygons: three triangles
(equilateral triangle, right isosceles, and the $(\pi/6,\pi/3,\pi/2)$ triangle),
or a square. This in turn implies that their are at most {\it four} faces in $P$
asymptotic to the vertex $v$.  So if $\mG^\prime$ is dual to $P$,
then the faces of $\mG^\prime$ corresponding to ideal vertices of
$P$ have either $3$ or $4$ sides.  Putting this together we see that
$\mG^\prime$ has the property that {\it all its faces are triangles
or squares}.

Finally, given such an $\mG^\prime$, with labels $m_e$ on the edges
(coming from the Coxeter graph), we can dualize $\mG^\prime$ to
obtain a polyhedron $P$.  We can assign to the edge $e^*$ of $P$
dual to the edge $e$ of $\mG^\prime$ the dihedral angle $\theta_e:=
\pi/m_e$.  This gives a labeled combinatorial polyhedron, to which
we can now apply Andreev's theorem, and find out whether it is
realizable in $\mH^3$.  This allows us to efficiently determine whether 
the original, abstract Coxeter group $\G$ is one to which the techniques 
of this paper apply.

Lastly, we point out that while these constraints are fairly
stringent, this nevertheless allows for infinitely many pairwise
non-isomorphic Coxeter groups $\G _P$ (see for instance Example
8 in Section 6 for infinitely many examples with polyhedron a 
combinatorial cube).

%%%%%%%%%%%%%%%%%%%%%%%%%%%%%%%%%%%%%%%%%%%%%%%%
\subsection{Isomorphism conjecture and splitting formulas}

The starting point for our computation is the Farrell-Jones
isomorphism conjecture, which predicts, for a group $G$, that the
natural map:
\begin{equation}
H_n^{G}(E_{\vc}G;\mathbb K\mathbb Z^{-\infty}) \rightarrow
H_n^{G}(*;\mathbb K\mathbb Z^{-\infty})  \cong K_n(\mathbb ZG)
\end{equation}
is an isomorphism for all $n$.  This conjecture is known to hold for
lattices in $O^+(3,1)$ for $n\leq 1$, by work of Farrell \& Jones
\cite{FJ1} in the cocompact case, and by work of Berkove, Farrell,
Juan-Pineda, and Pearson \cite{BFJP} in the non-cocompact case.
So to compute the (algebraic) right hand side, we can instead focus
on computing the (topological) left hand side.

Let us discuss the left hand side.  Farrell-Jones \cite{FJ1} established the
existence of an equivariant generalized homology theory (denoted
$H_*^?(-; \kz)$), having the property
that for any group $G$, and any integer $n$, the equivariant homology of
a point $*$ with trivial $G$-action satisfies $H_n^G(*; \kz)\cong K_n(\mathbb ZG)$.
Now given any $G$-CW-complex $X$, the obvious $G$-equivariant map
$X\rightarrow *$ induces a canonical homomorphism in equivariant homology:
$$H_n^G(X; \kz) \rightarrow H_n^G(*; \kz)\cong K_n(\mathbb ZG)$$
called the assembly map. The idea behind the isomorphism conjecture
is to find a ``suitable'' space $X$ for the above map to be an
isomorphism.  The space $X$ should be canonically associated to the
group $G$, and should be explicit enough for the left hand side to
be computable.  In the isomorphism conjecture, the space $E_{\vc}G$
that appears is any model for the classifying space for $G$-actions
with isotropy in the family of virtually cyclic subgroups.  We refer
the reader to the survey paper by L\"uck and Reich \cite{LR} for
more details.

Having explained the left hand side of the Farrell-Jones isomorphism
conjecture, let us now return to the groups for which we would like
to do computations.  Recall that we are considering groups $\G _P$,
associated to certain finite volume geodesic polyhedra $P\subset
\mH^3$. These groups are automatically lattices in $O^+(3,1)$.

Now for lattices $\G \leq O^+(3,1)$, Lafont-Ortiz established in
\cite[Cor. 3.3]{LO2} the following formula for the algebraic
$K$-theory:
\begin{equation}
H_n^{\G}(E_{\vc}(\G);\mathbb K\mathbb Z^{-\infty})
\cong H_n^{\G}(E_{\fin}(\G);\mathbb K\mathbb Z^{-\infty})
\oplus \bigoplus _{V\in \mathcal V}H_n^{V}(E_{\fin}(V)\rightarrow*).
\end{equation}
Let us explain the terms showing up in the above formula.  The left
hand side is the homology group we are interested in computing, and
coincides with the algebraic $K$-theory groups $K_n(\mZ \G)$ (as
equation (1) is an isomorphism).  The first term on the right hand
side is the equivariant homology of $E_{\fin}\G$, a model for the
classifying space for proper actions of $G$.  But it is well known
that for lattices in $O^+(3,1)$, such a model is given by the action
on $\mH^3$. For the second term appearing on the right hand side, we
have that:
\begin{itemize}
\item $\mathcal V$ consists of one representative $V$ from each conjugacy class in $\G$
of those infinite subgroups of the form $\stab_{\G}(\gamma)$, where $\gamma$ ranges over
geodesics in $\mH ^3$.
\item the groups $H_n^{V}(E_{\fin}(V)\rightarrow*)$ are the cokernels of the
assembly maps $H_n^{V}(E_{\fin}(V); \kz) \rightarrow H_n^{V}(*; \kz)$.
\end{itemize}
In view of the splitting formula, we merely need to analyze
the two terms appearing on the
right hand side of equation (2). 

The first term is computed via an
Atiyah-Hirzebruch type spectral sequence, which we will 
analyze in Sections 3 and 4. The remaining terms will be 
analyzed in our last Section 5.

%%%%%%%%%%%%%%%%%%%%%%%%%%%%%%%%%%%%%%%%%%%%%%%%%%%

\section{Lower algebraic $K$-theory of $D_n$, $D_n\times \mZ_2$, and $A_5 \times \mZ_2$}

In order to compute the term $H_*^{\G}(E_{\fin}(\G_P);\mathbb K\mathbb Z^{-\infty})$,
we will make use of an Atiyah-Hirzebruch type spectral sequence due to Quinn. The 
computation of the $E^2$-terms of the spectral sequence requires 
knowledge of the lower algebraic $K$-theory of  cell stabilizers for the
$\G_P$-action on $\mH ^3$. For many of the groups arising as cell
stabilizers, the lower algebraic $K$-theory is known (see \cite[Section 5]{LO2}).
The only lower algebraic $K$-groups we still need
to compute are those of dihedral groups $D_n$ (generic stabilizers of 1-cells),  
those of groups of the form $D_n\times \mZ_2$ (generic stabilizers of 0-cells), as well
as the $K_{-1}$ of the group $A_5 \times \mZ _2$.

\vskip 5pt

We recall that Carter \cite{C}
established that $K_n(\mZ G)=0$ for $n\leq -2$ whenever $G$ is a
finite group.  In particular, we will just focus on computing the
$K_{-1},\tilde K_0,Wh$ for the generic stabilizers $D_n$ and $D_n\times \mZ_2$. 
Among these, we provide easily computable number 
theoretic expressions for the $K_{-1}$ (Section 3.1) and for $Wh$ (Section 3.4). 
In contrast, the determination of $\tilde K_0$ is a classical, hard question; we provide 
a summary of what is known (Section 3.3). We also calculate the group
$K_{-1}(\mZ [A_5 \times \mZ _2])$ (Section 3.2).

%%%%%%%%%%%%%%%%%%%%%%%%%%%%%%%%%%%%%%%%%%%%%%%%%%%

\subsection{The negative $K$-theory $K_{-1}(\mathbb ZG)$}
A general recipe for computing the $K_{-1}$ of integral group rings of finite groups is 
provided by Carter \cite{C}. First recall that if $A$ is a simple artinian ring, it is isomorphic to $M_n(D)$ 
for some positive integer $n$ and 
division ring $D$, finite dimensional over its center $E$. That dimension $[D:E]$ is a square, and
the {\it Schur index} of $A$ equals $\sqrt{[D:E]}$. For a field $F$ and a finite group $G$, let $r_F$ 
denote the number of isomorphism classes of simple $FG$-modules. D. W. Carter \cite{C} proved
\[
K_{-1}(\mathbb ZG) \cong \mathbb Z^r \oplus (\mZ_2)^s
\]
where
\begin{equation}
r=1-r_{\mathbb Q} + \sum_{p\,|\, |G|} (r_{\mathbb Q_{p}} -r_{\mathbb F_{p}})
\end{equation}
where, as in all our summations, $p$ is prime; and $s$ is the number of simple components $A$ of $\mathbb QG$ with even Schur index but 
with $A_P$ of odd Schur index for each prime $P$ of the center of $A$ that divides $|G|$. We
now proceed to use Carter's formula to compute the $K_{-1}$ associated to the groups
$D_n, D_n\times \mZ_2$, and $A_5 \times \mZ _2$.

\vskip 10pt

Let us first consider the case of the groups $D_n$ and $D_n\times \mZ _2$. If $H$ is a subgroup of a group $G$, denote its index in $G$ by $[G : H].$
Suppose that $n$ is an integer exceeding 2, $\delta(n)$ is the number of (positive) divisors 
of $n$, and for each prime $p$, write $n=p^{\nu}\mu$ where $\mu \notin p\mathbb Z$. So $\nu = \nu_p(n)$
is the power to which $p$ divides $n$, and $\mu=\mu_p(n)$ is the non-$p$-part of $n$. 
Define
$$\sigma_p(n) =\sum_{d\mid \mu} [\,\mathbb Z^{\ast}_d : \langle  - \bar{1}, \bar{p} \rangle\, ] = \sum_{p \nmid d \mid n}  [\,\mathbb Z^{\ast}_d : \langle  - \bar{1}, \bar{p} \rangle\, ] $$
$$\tau(n) =\sum_{p \mid 2n} \nu_p(n)\sigma_p(n) = \sum_{p \mid n} \nu_p(n) \sigma_p(n)$$
We are now ready to state our:

\begin{Thm}
If $D_n$ is the dihedral group of order $2n$ and $\mZ_2$ is the cyclic group of order 2, then

\begin{enumerate}
\item $K_{-1}(\mathbb ZD_n) \cong \mathbb Z^{1-\delta(n) +\tau(n)},$

\vspace{.2cm}
\item $K_{-1}(\mathbb Z[D_n \times \mZ_2]) \cong \mathbb Z^{1-2\delta(n) + \sigma_2(n)+2\tau(n)}.$

\end{enumerate}
\end{Thm}

\begin{proof}

For the groups $G=D_n$ and $G=D_n \times \mZ_2$, we know 
that the simple components of $\mathbb QG$ are matrix 
rings over fields (see below); so each has Schur index 1, which forces $s=0$. This gives us the

\vskip 5pt

\noindent {\bf Fact 1:} For 
the groups $G=D_n$ and $G= D_n \times \mZ_2$, the group $K_{-1}(\mZ G)$ is torsion-free.

\vskip 5pt

So for these groups, we are left with having to compute the 
quantities $r_F$ where $F$ are various fields. Now for $F$ a field of characteristic 0, 
$FG$ is semisimple, and $r_F$ coincides with the number of simple components of $FG$ in its 
Wedderburn decomposition.

Let us denote by $\epsilon$ the number of conjugacy classes of reflections in $D_n$; so $\epsilon$ 
is 1 or 2,  according to whether $n$ is odd or even. As shown by Magurn \cite{Ma1},
\[
\mathbb QD_n \cong \bigoplus_{d\,|\,n,\, d>2} M_2(\mathbb Q (\zeta_d + \zeta^{-1}_d)) \oplus \mathbb Q^{2\epsilon};
\]
so $r_{\mathbb Q}=\delta(n)+\epsilon$. On the other hand, for the group $D_n \times \mZ_2$, we
know that $\mQ [D_n \times \mZ_2] \cong \mQ D_n \oplus \mQ D_n$, which immediately tells us that
for these groups, $r_{\mathbb Q}=2 \delta(n)+ 2 \epsilon$. We summarize these observations in our

\vskip 5pt

\noindent {\bf Fact 2:} For the groups $G=D_n$, we have that $r_{\mathbb Q}=\delta(n)+\epsilon$.
For the groups $G=D_n \times \mZ_2$, we have that $r_{\mathbb Q}= 2 \delta(n)+2 \epsilon$.

\vskip 5pt

To count simple $FG$-modules for an arbitrary field $F$ of characteristic $p$ (possibly $p=0$), 
we employ a theorem of S.\ D.\ Berman. Suppose $d$ is a positive integer and $d\neq0$ in $F$. 
Then $x^d-1$ has $d$ different roots in the algebraic closure $\bar{F}$, and these form a (necessarily 
cyclic) subgroup of $\bar{F}^{\ast}$. Say $\zeta_d$ is a generator -- a primitive $d^{\text{th}}$ root of 
unity over $F$. Now $F_d=F(\zeta_d)$ is a Galois extension of $F$, and each member of the Galois 
group $\at(F_d/F)$ is defined by its effect $\zeta_d \mapsto \zeta^{t}_d$, where $t$ is coprime to $d$. 
Sending such an automorphism to $\bar{t} \in \mathbb Z_d^{\ast}$ defines an embedding of the Galois 
group as a subgroup $T_d$ of $\mathbb Z^{\ast}_d$

With $p$ as above, an element $x \in G$ is $p$-{\sl regular} if $p$ does not divide the order of $x$. 
Suppose $m$ is the least common multiple of the orders of all $p$-regular elements of $G$. Then 
$m \neq0$ in $F$. Say $p$-regular elements $x,  y \in G$ are $F$-{\sl conjugate} if $x^t=gyg^{-1}$ 
for some $\bar{t} \in T_m$ and $g \in G$. This is an equivalence relation on the set of $p$-regular 
elements of $G$. Notice that $F$-conjugate $p$-regular elements have equal order, since each $t$ 
is coprime to $m$, hence to their orders. Berman established (see \cite[Thms. (21.5) and
(21.25)]{CR1}) that for $F$ is field of characteristic $p$ (possibly $p=0$) and $G$ a finite group, the number 
$r_F$ of isomorphism classes of simple $FG$-modules is equal to the number of $F$-conjugacy 
classes of $p$-regular elements of $G$. In view of this result and Carter's formula (3), we will now focus on counting the
$F$-conjugacy classes in the groups $D_n$ and $D_n \times \mZ_2$, where $F$ is either
a $p$-adic field $\mQ _p$ or a finite field $\mathbb F _p$. We first work over the field $\mQ _p$, 
and establish

\vskip 10pt

\noindent {\bf Fact 3:} For the groups $G=D_n$, we have that $r_{\mathbb Q_p}=(\nu_p(n) +1) \sigma_p(n) + \epsilon$.
For the groups $G=D_n \times \mZ_2$, we have that $r_{\mathbb Q_p}= 2(\nu_p(n) +1) \sigma_p(n) + 2\epsilon$.

\vskip 5pt

To see this, we first recall that for a prime $p$, the field $\mathbb Q_p$ of 
$p$-adic numbers has characteristic 0. Every element of 
a finite group $G$ is 0-regular, and hence the least common multiple $m$ of the orders of 
0-regular elements coincides with the minimum exponent of $G$. We now set $F=\mathbb Q_p$, 
and for each $m$, denote by $F_m=\mathbb Q_p(\zeta_m)$. Write 
$m=p^eq$ with $e \geq 0$ and $q$ an integer not divisible by $p$.

As shown by \cite[Chapter IV, Section 4]{Se1}, $\at(F_{p^e}/F)$ embeds as all $\mathbb Z^{\ast}_{p^e}$ 
and $\at(F_q/F)$ embeds as the cyclic subgroup $\langle \bar{p}\rangle$ in $\mathbb Z^{\ast}_{q}$. The 
former is deduced from 
$$[F_{p^e}:F]=\phi(p^e)=|\mathbb Z^{\ast}_{p^e}|,$$
and this, in turn, follows from the irreducibility in $F[x]$ of the cyclotomic polynomial 
$$p(x)=1+x^{p^{e-1}} + x^{2p^{e-1}}+ \cdots + x^{(p-1)p^{e-1}}.$$
This irreducibility comes from that of the Eisenstein polynomial $p(x+1)$ in $F[x]$. Now $F_q$ is an unramified 
extension of $F$, so $p$ remains prime in its valuation ring and $p(x+1)$ is still an Eisenstein polynomial in 
$F_q[x]$. So $p(x)$ is irreducible there, and 
$$[F_m : F_q]= [F_{q}(\zeta_{p^e}):F_q]=\phi(p^e).$$
Therefore $[F_m: F]=[F_{p^e}: F][F_q:F].$

Now $\zeta_{p^e} \zeta_q$ has order $m$. So we can choose $\zeta_m$ to be  $\zeta_{p^e} \zeta_q$, 
and each element of  $\at(F_m/F)$ is uniquely determined by its restrictions to $F_{p^e}$ and $F_q$. 
The resulting embedding
$$\at(F_m/F) \longrightarrow \at(F_{p^e}/F) \times \at(F_{q}/F)$$
must be an isomorphism, since the domain and codomain have equal finite size. Note also that, for 
$d$ dividing $m$, the restriction map
$$\at(F_m/F) \longrightarrow \at(F_{d}/F)$$
is surjective by the Extension Theorem of Galois theory. So the canonical map $\mathbb Z_m \rightarrow 
\mathbb Z_d$ takes $T_m$ onto $T_d$.

\vskip 5pt

Next, let us specialize to $G=D_n$, with $F$ still the $p$-adic field $\mathbb Q_p$. Then $m$ is the least common 
multiple of $2$ and $n$. For $\bar{t} \in T_m \leq \mathbb Z^{\ast}_m$, $t$ is coprime to $m$, so must be odd. 
This implies that for each reflection $b \in D_n$, $b^t=b$, telling us that the $\mathbb Q_p$-conjugacy class of 
$b$ coincides with its ordinary conjugacy class. Hence there are $\epsilon$ distinct $\mQ _p$-conjugacy
classes of reflections in $D_n$.

Now let us consider rotations in $D_n$. Each rotation $x \in D_n$ has order dividing $n=p^{\nu} \mu$, and
hence has order $p^id$ with 
$i \leq \nu$ and $d$ dividing $\mu$. Every rotation of order $p^id$ is uniquely expressible as a product 
$yz$ where $y$ and $z$ are rotations of orders $p^i$, $d$ respectively. If $x$ has this decomposition $yz$, 
then as $\bar{t}$ varies in $T_m$, $x^t=y^tz^t=y^jz^k$ where $(\bar{j}, \bar{k})$ runs through all pairs in 
$T_{p^i} \times T_d$. Since rotations $x \in D_n$ have conjugacy class $\{x, x^{-1}\}$, the 
$\mathbb Q_p-$conjugacy class of a rotation $x$ of order $p^id$ is the set of
$$|T_{p^i} \times T_d| = |\mathbb Z^{\ast}_{p^i} \times \langle \bar{p} \rangle| = \phi(p^i)| \langle \bar{p} \rangle|$$ 
elements $x^t$, together with their inverses.

If $\langle \bar{p} \rangle = \langle -\bar{1}, \bar{p} \rangle$ in $\mathbb Z^{\ast}_d$, each $-\bar{p}^u$ is $p^v$ 
for some integer $v$, and the set of powers
$$x^t=y^u z^{p^u}$$
is closed under inverses. Or if $\langle \bar{p} \rangle \neq \langle -\bar{1}, \bar{p} \rangle$, then $-\bar{p}^u \notin  
\langle \bar{p} \rangle$, and the set of powers $x^t$ does not overlap the set of $x^{-t}$. Either way the number of 
$\mathbb Q_p$-conjugacy classes of $x$ is $\phi(p^i)| \langle -\bar{1}, \bar{p} \rangle|$. There are $\phi(p^i) \phi(d)$ 
rotations of order $p^id$; so the number of $\mathbb Q_p$-conjugacy classes of them is the index
\[
\dfrac{\phi(d)}{| \langle -\bar{1}, \bar{p} \rangle|}=[\mathbb Z^{\ast}_d :  \langle -\bar{1}, \bar{p} \rangle].
\]
This number is independent of $i$. Since there are $\nu+1$ powers $p^i$, we obtain
\[
r_{\mathbb Q_p}= (\nu +1) \sum_{d\mid \mu}^{} [\mathbb Z^{\ast}_d :  \langle -\bar{1}, \bar{p} \rangle]+ \epsilon
=(\nu +1) \sigma_p(n) + \epsilon
\]
Finally, we observe that we have an isomorphism $\mQ _p[D_n \times \mZ_2] \cong \mQ_pD_n \oplus
\mQ_p D_n$. Recalling that the integer $r_{\mQ _p}$ can also be interpreted as the number of 
simple components in the Wedderburn decomposition of $\mQ _p G$, this tells us that for the group
$D_n \times \mZ_2$, we have $r_{\mQ _p} = 2(\nu_p(n) +1) \sigma_p(n) + 2\epsilon$. This concludes the
verification of Fact 3. 

\vskip 10pt

Finally, we work over the finite fields $\mF_p$, and establish

\vskip 5pt

\noindent {\bf Fact 4:} For the groups $G=D_n$, we have that $r_{\mF_p}= \sigma_p(n) + \epsilon$ ($p$ odd) 
and $r_{\mF_2} = \sigma_2(n)$.
For the groups $G=D_n \times \mZ_2$, we have that $r_{\mF_p}= 2 \sigma_p(n) + 2\epsilon$ ($p$ odd) 
and $r_{\mF_2} = \sigma_2(n)$.

\vskip 5pt

Let us first consider the case $G= D_n$. Set $F= \mathbb F_p$, and let $m$ be the least common multiple 
of the orders of the 
$p$-regular elements in $G$. The fields $F$, $F(\zeta_m)$ are finite and $\at(F(\zeta_m)/F)$ is cyclic, 
generated by the $p$-power map, since $|F|=p$. So $T_m =\langle \bar{p} \rangle \leq \mathbb Z^{\ast}_m$. 
Fix a reference rotation $a \in D_n$ of order $n$. 

The $p$-regular rotations are the rotations of order $d$ dividing $\mu_p(n)$, and so are of the form $\alpha^u$, where 
$\alpha=a^{n/d}$ and $u$ is coprime to $d$. Then $\alpha^u$, $\alpha^v$ are $\mathbb F_p$-conjugate if 
and only if
$$\alpha^{up^i} = \alpha^{v(-1)^j}$$
for some positive integers $i$, $j$. This just means $\bar{u} \equiv \bar{v}$ mod  $\langle -\bar{1}, \bar{p} \rangle$ 
in $\mathbb Z^{\ast}_d$. So the number of $\mathbb F_p$-conjugacy classes of $p$-regular rotations in $D_n$ is
\[
\sum_{d\mid \mu}  [\mathbb Z^{\ast}_d :  \langle -\bar{1}, \bar{p} \rangle]= \sigma_p(n).\]
Next let us consider reflections in $D_n$. If $p$ is an odd prime, reflections are also $p$-regular. Since $m$ is 
even, each $t$ with $\bar{t} \in T_m \leq \mathbb Z^{\ast}_m$ is odd. So we see that for reflections, 
$\mF_p$-conjugacy coincides with ordinary conjugacy. In particular, for odd primes, we see that
the number of $\mF _p$-conjugacy classes of reflections is $\epsilon$. On the other hand, 
reflections are not 2-regular. Putting this together, we obtain that
\[
r_{\mathbb F_p}= \sigma_p(n) +\epsilon \hskip 10pt (p \;\; \text{odd}), \;\; \text{and}\;\; r_{\mathbb F_2}= \sigma_2(n).
\]
which establishes the first part of Fact 4.

\vskip 5pt

Now let us consider groups of the form $G=D_n\times \mZ_2$, where $\mZ_2$ is the cyclic group $\{1,c\}$ generated
by $c$ or order $2$. 
In the group $D_n\times \mZ_2$, the subgroup $D_n \times \{1\}$ is a copy of $D_n$. Fix a reference rotation 
$a\in D_n$ of order $n$ and a reference reflection
$b\in D_n$. Then in $D_n\times \mZ_2$ we have 
rotations $(a^i, 1)$, reflections $(a^ib, 1)$, 
co-rotations $(a^i,c)$  and co-reflections $(a^ib,c)$. For $x \in D_n$, the conjugates in $D_n \times \mZ_2$ of 
$(x,1)$ (resp. $(x, c)$) are the elements $(y, 1)$ (resp. $(y, c)$) with $x$ conjugate to $y$ in $D_n$. The least 
common multiple $m$ of orders of $p$-regular elements is the same for $D_n$ and $D_n \times \mZ_2$; so the 
exponents $t$ with $\bar{t} \in T_m$ are the same for both these groups.

For $p$ odd, we have that $m$ is even, hence $t$ are odd, and $(x,1)^t=(x^t, 1)$, while $(x, c)^t=(x^t, c)$. 
The $p$-regular elements in $D_n\times \mZ_2$ are the reflections, the co-reflections, and the $p$-regular 
rotations and co-rotations. So for $p$ odd, the number of 
$\mathbb F_p$-conjugacy classes in $D_n \times \mZ_2$ is double the number of $\mF _p$-conjugacy 
classes in the corresponding $D_n$:
$$r_{\mathbb F_p} = 2 \sigma_p(n) + 2 \epsilon  \hskip 10pt (p \;\; \text{odd}).$$
On the other hand, the only 2-regular elements in $D_n \times \mZ_2$ are rotations. This implies that the 
$\mathbb F_2$-conjugacy classes are the same for $D_n \times \mZ_2$ as for $D_n$, giving 
us $r_{\mathbb F_2} =  \sigma_2(n)$. This completes the verification of the second statement in
Fact 4.

\vskip 10pt

Finally, let us apply Carter's formula to complete the proof of our theorem. For
groups of the form $G=D_n$, we have from Fact 1 that $K_{-1}(\mZ D_n)$ is free abelian. Furthermore,
the rank of $K_{-1}(\mZ D_n)$ is given by equation (3). Combining Fact 3 and Fact 4, we see that
\[
r_{\mathbb Q_p} - r_{\mathbb F_p}=
\begin{cases}
\nu_p(n) \sigma_p(n), & p\neq 2 \\
\nu_2(n) \sigma_2(n) + \epsilon, & p=2.
\end{cases}
\]
Summing over all primes dividing $|G|=2n$, we see that:
\[
\sum_{p\mid 2n} (r_{\mathbb Q_p} - r_{\mathbb F_p})= \bigg( \sum _{p\mid 2n}\nu_p(n) \sigma_p(n) \bigg) 
+ \epsilon = \tau(n) + \epsilon.
\]
Substituting this expression into equation (3), and substituting in the calculation of $r_{\mQ}$ from 
our Fact 2, we see that the rank of $K_{-1}(\mZ D_n)$ is given by:
\[
r= 1 - (\delta(n) +\epsilon)-(\tau(n) + \epsilon) = 1 -\delta(n) + \tau(n),
\]
which establishes part (1) of our Theorem 1.

\vskip 10pt

Similarly, for groups of the form $G= D_n \times \mZ_2$, we again have from Fact 1 that 
$K_{-1}(\mZ D_n)$ is free abelian. Combining Fact 3 and Fact 4, we see that for these groups
\[
r_{\mathbb Q_p} - r_{\mathbb F_p}=
\begin{cases}
2 \nu_p(n) \sigma_p(n), & p\neq 2 \\
2 \nu_2(n) \sigma_2(n) + \sigma_2(n) + 2\epsilon, & p=2.
\end{cases}
\]
Summing over all primes dividing $|G|=4n$ (which coincides with the primes dividing $2n$), we see that:
\[
\sum_{p\mid 4n} (r_{\mathbb Q_p} - r_{\mathbb F_p})= 2 \bigg( \sum _{p\mid 2n}\nu_p(n) \sigma_p(n) \bigg) + \sigma_2(n) +  2 \epsilon 
= 2\tau(n) + \sigma_2(n) + 2 \epsilon.
\]
Finally, substituting these into equation (3), and substituting the calculation of $r_{\mQ}$ from 
our Fact 2, we obtain that $K_{-1}(\mathbb Z[D_n \times \mZ_2])$ has rank
$$r=1-(2\delta(n) + 2\epsilon) + (2 \tau(n) + \sigma_2(n) + 2\epsilon) = 1- 2\delta(n) + \sigma_2(n) +2\tau(n),$$ 
establishing part (2) of our theorem. This concludes the proof of Theorem 1.
\end{proof}

As an application of this result, we can easily determine the $K_{-1}$ of dihedral groups $D_n$ when
$n$ has few divisors. For example, we have:

\begin{corollary}
For $p$ an odd prime, set $r=[\mathbb Z^{\ast}_p :  \langle -\bar{1}, \bar{2} \rangle]$. Then
we have 
\begin{itemize}
\item $K_{-1}(\mathbb ZD_{2p}) \cong K_{-1}(\mathbb Z[D_p \times\mZ_2]) \cong \mathbb Z^r$,
\item $K_{-1}(\mathbb Z[D_{2p}\times \mZ_2]) \cong \mathbb Z^{3r}$.
\end{itemize}
\end{corollary}

\begin{proof}
From our theorem, we have that the rank of $K_{-1}(\mathbb ZD_{2p})$ is given by:
$$1-\delta (2p) + \tau (2p).$$ 
We have that $\delta(2p)=4$, while $\tau(2p)=
\nu_2(2p) \sigma_{2}(2p) + \nu_p(2p) \sigma _{p}(2p)$. Since $\nu_2(2p)=\nu_p(2p)=1$, we are left
with computing the integers $\sigma_{2}(2p), \sigma _{p}(2p)$. But it is easy to
verify that $\sigma _{p}(2p)=2$, while $\sigma _{2}(2p)= 1+ [\mathbb Z^{\ast}_p :  
\langle -\bar{1}, \bar{2} \rangle]=1+r$. This gives us $\tau(2p)=3+r$; 
substituting in these values gives the first
statement in our corollary.

Similarly, we know that the rank of $K_{-1}(\mathbb Z[D_{2p}\times \mZ_2])$ is
given by:
$$1- 2 \delta(2p) + \sigma_{2}(2p)+2\tau(2p).$$
Substituting in the values computed above, we find that the rank is given by
$$1 - 2(4) + (1+r) + 2(3+r) = 3r$$
which establishes the second statement in the corollary.
\end{proof}

For a concrete example, we see for instance that when $p=3$, $r=1$, then
$K_{-1}(\mathbb ZD_6) \cong \mathbb Z$ and $K_{-1}(\mathbb Z[D_6 \times \mZ_2]) 
\cong \mathbb Z^3$ (see \cite[Section 5.1]{LO2} for an alternate computation of this group). 

As another application, we can completely classify dihedral groups 
(and products of dihedral groups with $\mZ_2$) whose $K_{-1}$ vanishes:

\begin{corollary}
$K_{-1}(\mathbb ZD_n)=0$ if and only if $n$ is a prime power, and 
$K_{-1}(\mathbb Z[D_n \times \mZ_2])=0$ if and only if $n$ is a power of 2.
\end{corollary}

\begin{proof}
Note that 
\[
\sigma_p(n) = \sum_{d\mid \mu} [\mathbb Z^{\ast}_d :  \langle -\bar{1}, \bar{p} \rangle] \geq  \sum_{d\mid \mu} 1 
= \delta(\mu)= \dfrac{\delta(n)}{\delta(p^{\nu})}=\dfrac{\delta(n)}{1+ \nu}.\]
So
\[
\tau(n) \geq  \sum_{p\mid n} \nu_p(n) \bigg( \dfrac{\delta(n)}{1+ \nu_p(n)}\bigg)= \delta(n)  
\sum_{p\mid n}\bigg(1 - \dfrac{1}{1+ \nu_p(n)}\bigg) 
\geq \delta(n) \sum_{p\mid n} \dfrac{1}{2} = \delta(n)\dfrac{g}{2},
\]
where $g$ is the number of primes that divide $n$. Then the rank of $K_{-1}(\mathbb ZD_n)$ is
\[
1 - \delta(n) + \tau(n) \geq 1 + \delta(n)\Big(\dfrac{g}{2} -1\Big).
\]
So $K_{-1}(\mathbb ZD_n)$ does not vanish if $n$ is not a prime power. If $n=p^e$ for $p$ 
prime and $e \geq 1$, then $\tau(n)= \nu_p(n) \sigma_p(n)= (e)(1) =e$ and $\delta(n)=e+1$, 
which gives us $K_{-1}(\mathbb ZD_{p^e})=0$. 

\vskip 5pt

Next we consider the case of groups $D_n \times \mZ_2$. For arbitrary $n$,
\[
\sigma_2(n) + 2 \tau(n) \geq \delta(n) \bigg[ \dfrac{1}{1+ \nu_2(n)} +g\bigg] > \delta(n) g.
\]
This allows us to estimate from below the rank of $K_{-1}(\mathbb Z[D_n \times \mZ_2])=$
$$1 - 2\delta(n) + \sigma_2(n) + 2 \tau(n) > 1+ \delta(n) (g-2).$$
So $K_{-1}(\mathbb Z[D_n \times \mZ_2])$ doesn't vanish if $n$ is not a prime power. If $n=p^e$ 
for $p$ an odd prime and $e \geq 1$, its rank is given by
$$1- 2(e+1) + \sigma_2(n) + 2e = \sigma_2(n)-1.$$
But for such an $n$, $\nu_2(n)=0$ and $\delta (n)=1+e$, giving us
$$\sigma_2(n) \geq \dfrac{\delta(n)}{1+\nu_2(n)} = e+1 \geq 2.$$
This implies that $K_{-1}(\mathbb Z[D_{p^e} \times \mZ_2])$ does not vanish. 
However, if $n=2^e$ with $e \geq 1$, the rank is
$$1- 2(e+1) +1 +2e=0$$
concluding the proof of Corollary 3.
\end{proof}

%%%%%%%%%%%%%%%%%%%%%%%%%%%%%%%%%%%%%%%%%%%%%%%%%

\subsection{The lower $K$-group $K_{-1}(\mathbb Z[A_5\times \mZ_2])$.}

Recall that the group $A_5\times \mZ_2$ is one of the three ``exceptional'' groups
that arise as stabilizers of vertices of $P$, along with $S_4$ and $S_4\times \mZ_2$.
The lower $K$-groups of the last two groups have all been previously 
computed (see \cite[Section 5]{LO2}). We implement the method discussed
in the previous section to show the following:

\begin{example} 
$K_{-1}(\mZ [A_5\times \mZ_2]) \cong \mZ ^2$.
\end{example} 

\begin{proof}
We first recall that the group algebra $\mQ A_5$ decomposes
into simple components as follows:
$$\mQ A_5 \cong \mQ \oplus M_3(\mQ\big(\sqrt{5})\big) \oplus M_4(\mQ) \oplus M_5(\mQ).$$
Since $\mQ [A_5\times \mZ_2] \cong \mQ A_5 \oplus \mQ A_5$, we see that the
Schur indices of all the simple components in the Wedderburn decomposition of 
$\mQ [A_5\times \mZ_2]$ are equal to 1. Carter's result 
\cite{C} now tells us that $K_{-1}(\mZ [A_5\times \mZ_2])$ is torsion-free, 
and from equation (3), the rank is given by
\begin{equation}
r=1-r_{\mathbb Q} + (r_{\mathbb Q_{2}} -r_{\mathbb F_{2}}) + (r_{\mathbb Q_{3}} -r_{\mathbb F_{3}}) + 
(r_{\mathbb Q_{5}} -r_{\mathbb F_{5}}).
\end{equation}
We now proceed to compute the various terms appearing in the above expression.

Recall that for $F$ a field of characteristic $0$, $r_F$ just counts the number of 
simple components in the Wedderburn decomposition of the group algebra $F[A_5\times \mZ_2]$. 
From the discussion in the previous paragraph, we have that
$$\mQ [A_5\times \mZ_2] \cong \mQ ^2 \oplus M_3\big(\mQ(\sqrt{5})\big)^2 \oplus M_4(\mQ)^2 
\oplus M_5(\mQ)^2.$$
yielding $r_{\mQ}=8$. Now by tensoring the above splitting with $\mQ _p$, we obtain:
$$\mQ_p [A_5\times \mZ_2] \cong \mQ_p ^2 \oplus M_3\big(\mQ_p \otimes _{\mQ} \mQ(\sqrt{5})\big)^2 
\oplus M_4(\mQ_p)^2 \oplus M_5(\mQ_p)^2.$$
The second term is isomorphic to $M_3(\mQ_p(\sqrt{5}))^2$ for $p=2,3$ and 5, since 5 is not 
a square mod 8, 3, or 25, hence not a square in $\mathbb Q_p.$ In particular, for each of the 
primes $p=2,3,5$, we obtain that $r_{\mQ_2}=r_{\mQ_3}=r_{\mQ _5}=8$.

\vskip 5pt

Next let us consider the situation over the finite fields $\mF_2, \mF_3, \mF _5$. We first recall that the 
integer $r_{\mF _p}$ counts the number of $\mF_p$-conjugacy classes
of $p$-regular elements. Recall that an $\mF _p$-conjugacy class
of an element $x$ is the union of ordinary conjugacy classes of certain specific powers of $x$, where
the powers are calculated from the Galois group associated to a finite extension of the field $\mF _p$.

For $p=2$, we note that elements in $A_5 \times \mZ _2$ are $2$-regular precisely if they have order $1$, $3$ or $5$.
There is a single conjugacy class of elements of order one (the identity element). The elements of
order $3$ form a single conjugacy class inside $A_5\times \mZ_2$. Finally, the elements of order
$5$ form {\it two} conjugacy classes in $A_5\times \mZ_2$; representatives for these two conjugacy
classes are given by $g=\big((abcde), 1\big)$, and by $g^2$. So there will be either one or two $\mF_2$-conjugacy
classes of elements of order $5$. To 
determine the specific powers, we note that the minimal exponent of $A_5 \times \mZ_2$ equals
$30=2\cdot 15$. The powers of $x$ are given by considering the Galois group of the extension 
$Gal\big(\mF_2(\zeta _{15})/\mF_2\big)$, viewed as elements of $\mZ _{15}^*$. Since the Galois
group is generated by squaring, we see that the Galois group is cyclic of order 4, given by the
residue classes $\{\bar 1, \bar 2, \bar 4, \bar 8\} \subset \mZ _{15}^*$. In particular, since $\bar 2$
lies in the Galois group, we see that $g$ and $g^2$ lie in the same $\mF _2$-conjugacy class, 
implying that there is a {\it unique} $\mF_2$-conjugacy class of elements of order 5. We conclude
that there are three $\mF_2$-conjugacy classes of $2$-regular elements, giving $r_{\mF_2}=3$.

For $p=3$, the elements in $A_5 \times \mZ _2$ which are $3$-regular have order $1$, $2$, $5$, or
$10$. Since the minimal exponent of the group is $30= 3\cdot 10$, we look at the Galois group 
associated to the field extension $\mF_3 (\zeta _{10})$. Elements in the Galois group are generated
by the third power, giving us that $Gal(\mF_3(\zeta _{10})/\mF_3) = \{\bar 1, \bar 3, \bar 7, \bar 9\} \subset
\mZ _{10}^*$. In particular, the $\mF_3$-conjugacy class of any element $x\in A_5\times \mZ_2$ is 
the union of the conjugacy classes of the elements $x, x^3, x^7$, and $x^9$. Now we clearly have
a unique $\mF _3$-conjugacy class of elements of order one. For elements of order $2$, there are
{\it three} distinct (ordinary) conjugacy classes of elements of order two; each of these ordinary conjugacy
class is also an $\mF _3$-conjugacy class. Next we note that there are {\it two} conjugacy classes 
of elements of order $5$, but these two ordinary conjugacy classes are part of a single $\mF_3$-conjugacy class since the 7$^{th}$ power of such an element is its square. Finally, we observe that there are {\it two} conjugacy classes of elements
of order $10$, but these again give rise to a single $\mF _3$-conjugacy class. We conclude that
overall there are six $\mF _3$-conjugacy classes of $3$-regular elements, giving $r_{\mF_3}=6$.

Finally, for $p=5$, the elements in $A_5\times \mZ_2$ which are $5$-regular have order $1$, $2$, $3$,
or $6$. Since the minimal exponent of the group is $30=5\cdot 6$, we need to look at the Galois group
associated to the field extension $\mF_5(\zeta _ 6)$. Elements in the Galois group are generated by the
fifth power, giving us that $Gal(\mF_5(\zeta _{6})/\mF_5) = \{\bar 1, \bar 5\} \subset
\mZ _{6}^*$. This yields that the $\mF_5$-conjugacy class of an element $x\in A_5\times \mZ_2$ of order
five is the union of the ordinary conjugacy classes of $x$ and of $x^5$. Now we have a single $\mF_5$-conjugacy
class of elements of order one. For elements of order $2$, we have {\it three} ordinary conjugacy classes of
elements; but each of these also forms a single $\mF _5$-conjugacy class (since for these elements, $x^5=x$).
For elements of order $3$, we have a single ordinary conjugacy class of such elements, which also form
a single $\mF_5$-conjugacy class. Finally, for elements of order $6$, we also have a single ordinary conjugacy
class of such elements, which hence also form a single $\mF_5$-conjugacy class. We conclude that there
is a total of six $\mF_5$-conjugacy classes of $5$-regular elements, and hence $r_{\mF_5}=6$.

To conclude, we substitute in our calculations into the expression in equation (4) for the rank
of $K_{-1}(\mZ [A_5\times \mZ_2])$, giving us:
$$r= 1 - 8 + (8 - 3) + (8 - 6) + (8 - 6) = 2$$
completing our calculation for Example 4.
\end{proof}

%%%%%%%%%%%%%%%%%%%%%%%%%%%%%%%%%%%%%%%%%%%%%%%%%%%
\subsection{The class group $\tilde K_{0}(\mathbb ZG)$}

The group $\tilde K_0(\mathbb ZG)=K_0(\mathbb ZG)/\langle [\mathbb ZG]\rangle$ is closely related to the 
ideal class group of the ring of algebraic integers $R$ in a number field $F$ (that is, a field $F$ with 
$[F: \mathbb Q]$ finite).

The {\sl ideal class group} $Cl(R)$ is the group of $R$-linear isomorphism classes of non-zero ideals of $R$, 
under multiplication of ideals: $(I)(J)=(IJ)$. The identity  element is the class $(R)$ of non-zero principal ideals; so 
$Cl(R)$ measures the deviation of $R$ from being a principal ideal domain.

Suppose $A$ is a finite dimensional $\mathbb Q$-algebra. A $\mathbb Z$-{\sl order} in $A$ is a subring 
$\Lambda$ of $A$ that is finitely generated as a $\mathbb Z$-module and spans $A$ as a $\mathbb Q$-vector 
space. For each prime number $p$, the set $\mathbb Z - p \mathbb Z$ is a submonoid of $\mathbb Z$ under 
multiplication; the local ring
\[
S^{-1}\mathbb Z =\Big\{\dfrac{a}{n} : a \in \mathbb Z, n \in S \Big\}
\]
is denoted by $\mathbb Z_{(p)}$. For any $\Lambda$-module $M$, the localization $M_{(p)}=S^{-1}M$ is a 
$\Lambda_{(p)} = S^{-1}\Lambda$-module. We say $M$ is {\sl locally free} if there is an integer $n \geq 0$ so 
that, for all primes $p$, $M_{(p)}= \Lambda^n_{(p)}$ as a $\Lambda_{(p)}$-modules. This $n$ is the rank 
$rk(M)$ of $M$. By \cite[Lemma 6.14]{Sw}, finitely generated locally free $\Lambda$-modules are projective; 
so they generate a subgroup $LF(\Lambda)$ of $K_0(\Lambda)$. There is a surjective group homomorphism
\[
rk:LF(\Lambda) \rightarrow \mathbb Z, \hskip 10pt  [P]-[Q] \mapsto rk(P)- rk(Q).
\]
Its kernel is the {\sl locally free class group} $Cl(\Lambda)$ of the $\mathbb Z$-order $\Lambda$. Since $rk$ is 
split by sending $1 \in \mathbb Z$ to $[\Lambda] \in LF(\Lambda)$, there is an isomorphism $Cl(\Lambda) 
\cong LF(\Lambda)/\langle [\Lambda] \rangle$. As shown in \cite[39.13]{CR2}, $Cl(\Lambda)$ is a finite group.

In the classical case where $A$ is a number field $F$ and $\Lambda$ is its ring of algebraic integers $R$, 
$Cl(R)$ as defined above coincides with the classical ideal class group $Cl(R)$, and with $\tilde K_0(R)$. 
Its order $h(F)$ is called the {\sl class 
number} of $F$.

Our focus is the case $A=\mathbb QG$, $\Lambda=\mathbb ZG$, for a finite group $G$. Swan proved (see 
\cite[32.11]{CR1}) that every finitely generated projective $\mathbb ZG$-module is locally free; so 
$Cl(\mathbb ZG) \cong \tilde K_0(\mathbb ZG)$. But group rings and rings of algebraic integers are special 
in this respect; $Cl$ and $\tilde K_0$ differ for $\mathbb Z$-orders in general. Between these two, it is $Cl$ 
that inherits the properties of $K_0$:
$$Cl(\Lambda_1 \oplus \Lambda_2) \cong Cl(\Lambda_1) \oplus Cl(\Lambda_2)$$ 
by \cite{RU2}, and
$$Cl(M_n(R)) \cong Cl(R)$$ 
by \cite[36.6]{Re}.

If $f: A_1 \rightarrow A_2$ is a $\mathbb Q$-algebra homomorphism carrying a $\mathbb Z$-order $\Lambda_1$ 
into a  $\mathbb Z$-order $\Lambda_2$, the map $K_0(f): K_0(\Lambda_1) \rightarrow K_0(\Lambda_2)$ induces 
a group homomorphism
 $$Cl(f): Cl(\Lambda_1) \rightarrow Cl(\Lambda_2)$$
making $Cl$ a functor. If $i: \Lambda \rightarrow \Lambda'$ is the inclusion of $\Lambda$ into a maximal 
 $\mathbb Z$-order $\Lambda'$ of $A$ containing $\Lambda$, the map $Cl(i)$ is surjective; its kernel $D(\Lambda)$ 
 is known as the {\sl kernel group} of $\Lambda$, and up to isomorphism $D(\Lambda)$ is independent of the choice 
 $\Lambda'$ (see \cite{J}).  So we have a (not necessarily split) short exact sequence
\begin{equation}
0 \rightarrow D(\Lambda) \rightarrow Cl(\Lambda) \xrightarrow{Cl(i)} Cl(\Lambda') \rightarrow 0.
\end{equation}
So for a finite group $G$ and associated $\Lambda := \mZ G$, understanding the group 
$\tilde K_0(\mZ G)\cong Cl(\Lambda)$ boils down to understanding the groups 
$Cl(\Lambda ^\prime)$ and $D(\Lambda)$ and the way these fit together.

\vskip 10pt

Let us now specialize to the case of dihedral groups.
For $G$ the dihedral group $D_n$, the isomorphism of $\mathbb Q$-algebras
\[
\mathbb QD_n \cong \bigoplus_{d\,|\,n, d>2} M_2(\mathbb Q (\zeta_d + \zeta^{-1}_d)) \oplus \mathbb Q^{2\epsilon}
\]
carries $\mathbb ZD_n$ into the maximal order
\[
\Lambda' \cong  \bigoplus_{d\,|\,n, d>2} M_2(\mathbb Z [\zeta_d + \zeta^{-1}_d]) \oplus \mathbb Z^{2\epsilon}.
\]
Since $Cl(\mathbb Z)=0$, we obtain
\[
Cl(\Lambda') \cong  \bigoplus_{d\,|\,n, d>2} Cl(\mathbb Z [\zeta_d + \zeta^{-1}_d]).
\]
These summands, and the ideal class groups $Cl(\mathbb Z [\zeta_d])$, have been studied since the 
19th century work of Kummer and Dedekind. They remain difficult to compute, and their orders 
\[
h_d = h\big(\mathbb Q(\zeta_d)\big) = |Cl(\mathbb Z [\zeta_d])|
\]
\[
h^+_d = h\big(\mathbb Q(\zeta_d+ \zeta^{-1}_d)\big) = |Cl(\mathbb Z [\zeta_d+ \zeta^{-1}_d])|
\]
are still topics of active research.

A prime $p$ is {\sl regular} if $p$ does not divide $h_p$, and {\sl semiregular} if $p$ does not divide $h^+_p$. 
A conjecture originally discussed by Kummer, but currently known as Vandiver's Conjecture, is that all primes are 
semiregular.  This has been verified for all primes less than 125,000 (see \cite{Wag}). The smallest  irregular prime 
is 37. Note that the order of $Cl(\Lambda')$ in the case $G=D_n$ is $\prod_{d} h^+_d$ for $d |n$, $d>2$. If $d|n$, 
then $h^+_d | h^+_n$ (see \cite{Le}); so $Cl(\Lambda')=0$ if and only if $h^+_n = 1$. Computer calculations show 
that $h_n^+=1$ for all $n \leq 71$ (see \cite{Li}).

\vskip 5pt

The kernel group $D(\mathbb ZD_n)$ vanishes for $n$ a prime \cite{GRU}, and for $n$ a power of a regular prime 
\cite{FKW}, \cite{K}. Each surjective group homomorphism  $G \rightarrow H$ induces a surjective homomorphism 
$D(\mathbb ZG) \rightarrow D(\mathbb ZH)$ by \cite{RU2}. So if $d|n$, $D(\mathbb ZD_n)$ maps onto  
$D(\mathbb ZD_d)$. According to \cite[Theorem 5.2]{EM2}, $D(\mathbb ZD_{p^2}) \cong (\mZ_p)^d$ for all 
semiregular primes $p$, where $d > 1$ when $p$ is irregular. So $D(\mathbb ZD_n)\neq0$ when $n$ is divisible by 
the square of an irregular, semiregular prime. Of course we can remove the ``semiregular" condition if Vandiver's 
Conjecture is true.

Also in \cite{EM2}, $D(\mathbb ZD_n)$ is shown to have even order if (a) $n$ is divisible by three different odd 
primes, (b) $n$ is divisible by 4 and two different odd primes, or (c) $n$ is divisible by two different primes in $1+ 4\mZ$ (in a parallel 
result by \cite{Le}, $h_n^+$ is even if $n$ is divisible by three distinct primes in $1+ 4 \mZ$). Another result in 
\cite{EM2} is the proof that $D(\mathbb ZD_n)=0$ for all $n <60$.

\vskip 10pt

Next let us consider the situation for groups of the form $G=D_n \times \mZ_2$. Letting 
$\Lambda=\mZ [D_n\times \mZ_2]$, we again exploit the
short exact sequence (5). Let $c \in \mZ_2$ denote the non-trivial element in the cyclic group of order 2. 
Since $\mQ[D_n \times \mZ_2] \cong \mQ D_n 
\oplus \mQ D_n$ by an isomorphism ($c \rightarrow (1,-1)$) taking $\mZ[D_n \times \mZ_2]$ into 
$\mZ D_n \oplus \mZ D_n$, 
the $Cl(\Lambda')$ vanishes for $G=D_n \times \mZ_2$ if and only if it vanishes for the corresponding $D_n$.

\vskip 5pt

In contrast to $Cl(\Lambda ^\prime)$, the computation of the kernel group $D(\Lambda)$ is much more involved.
First of all, let us consider the case where $n$ is a power of $2$.
By \cite[Example 6.9]{OT}, we know that $D(\mathbb Z[D_{2^r}\times \mZ _2]) \cong \mZ _{2^r}$. In fact, 
by \cite{T}, this group is the ``Swan subgroup" generated by the class of the ideal $I$ in 
$\mathbb Z[D_{2^r} \times \mZ _2]$ generated as an ideal by 5 and the sum of the elements in 
$D_{2^r} \times \mZ_2$. For an alternative generator of this group, 
consider the cartesian square
\[
\xymatrix@C=20pt@R=30pt{
\mathbb Z[D_n \times \mZ_2] \ar[d] \ar[r] & \mathbb Z[D_n]
     \ar[d] \\
\mathbb Z[D_n] \ar[r] & \mathbb F_2[D_n]}
\]
(with the left map $c \mapsto 1$, top map $c \mapsto -1$ and remaining maps reduction mod $2$). Now by \cite{RU2},
the corresponding $K$-theory Mayer Vietoris 
sequence restricts to an exact sequence:
\[
K_1(\mZ D_n)  \rightarrow  K_1(\mathbb F_2[D_n]) \xrightarrow{\partial} D(\mathbb Z[D_n \times \mZ_2])  
\rightarrow(D(\mathbb Z[D_n]))^2  \rightarrow  0.
\]
Now in the special case where $n=2^r$, the term $(D(\mathbb Z[D_n]))^2$ vanishes. 
Through computations of $(\mathbb F_2D_{2^r})^{\ast}$ 
one can show $D(\mathbb Z[D_{2^r} \times \mZ_2])$ is also generated by $\partial(1+b+ab)$. 
Finally, since $D(\mathbb Z[D_n \times \mZ_2])$ maps onto $D(\mathbb Z[D_{2^r} \times \mZ_2])$ 
if $2^r | n$, we see that $D(\mZ[D_n \times \mZ_2])\neq 0$ for $n$ even.

Now for $n$ odd, we know that $D_n \times \mZ_2 \cong D_{2n}$. If $p$ is an odd prime, $D(\mathbb Z[D_{2p}])$ 
is the cokernel of the map $R^{\ast} \rightarrow (R/2R)^{\ast}$, where $R=\mathbb Z [\zeta_d + \zeta^{-1}_d]$ (see 
\cite[50.14]{CR2}). Generally , if $n$ is odd, the first map in the above sequence factors as the reduced norm, followed 
by reduction mod 2:
\[
\bigoplus_{d|n, d>2} \mathbb Z [\zeta_d + \zeta^{-1}_d]^{\ast} \oplus \mZ[b]^{\ast} \rightarrow \bigoplus_{d|n, d>2} 
\bigg(\dfrac{\mathbb Z [\zeta_d + \zeta^{-1}_d]}{2}\bigg)^{\ast} \oplus \mathbb F_2[b]^{\ast}
\]
followed by an isomorphism to $K_1(\mathbb F_2D_n)$. So for odd $n$, if the kernel of the surjective map
\[
D(\mathbb Z[D_{2n}]) \rightarrow D(\mathbb Z[D_n])^2
\]
is to be zero, then $R^{\ast} \rightarrow (R/2R)^{\ast}$ must be surjective for all $R=\mathbb Z [\zeta_d + \zeta^{-1}_d]$ 
with $d|n$, $d>2$.

\vskip 10pt

The simplest conclusion we can draw about $\tilde K_0(\mZ D_n)$ is that it vanishes for $n < 60$. For a regular prime 
$p$, $\tilde K_0(\mZ D_{p^r})$ vanishes whenever $h^+_{p^r}=1$ (which may be true for all $r$ but is only known to 
be so for $\phi(p^r) \leq 66$). And $\tilde K_0(\mZ D_n)$ has even order if $n$ is divisible by too many different primes 
$p$. Computer calculations are now accessible for $\tilde K_0(\mZ G)$ for groups $G$ of modest size \cite{BB}.

%%%%%%%%%%%%%%%%%%%%%%%%%%%%%%%%%%%%%%%%%%%%%%%%%%%

\subsection{The Whitehead group $Wh(G)$}

From \cite{Ba1} and \cite{Wa} we know
\[
 K_1(\mZ G) \cong \pm G^{ab} \oplus SK_1(\mZ G) \oplus \mZ ^{r-q}
\]
where $SK_1(\mZ G)$ is finite and $r$ and $q$ are the numbers $r_{\mathbb R}$, $r_{\mQ}$ of simple components of 
$\mathbb RG$, $\mQ G$ respectively. From Berman's Theorem,  $r_{\mathbb R}$ is the number of conjugacy classes 
of unordered pairs $\{x, x^{-1}\}$ with $x \in G$, and $r_{\mQ}$ is the number of conjugacy classes of cyclic subgroups 
of $G$. Furthermore
$$Wh(G) = K_1(\mathbb ZG)/\{\pm G^{ab}\} =  SK_1(\mathbb ZG) \oplus \mZ ^{r-q}.$$
For $G$ a dihedral group $D_n$ with $\epsilon$ conjugacy classes of reflections ($\epsilon=$1 or 2 according
to whether $n$ is odd or 
even), we computed in Section 3.1 that $q=\delta(n) + \epsilon$ where $\delta(n)$ is the number of divisors of $n$. 
Counting conjugacy classes of pairs  $\{x, x^{-1}\}$ with $x \in D_n$, we find $r=(n+3\epsilon)/2$. So $K_1(\mZ D_n)$ 
and $Wh(D_n)$ have rank $(n+ \epsilon)/2 -\delta(n)$. Now $(D_n)^{ab} \cong (\mZ_2)^\epsilon$, and by \cite{Ma1}, 
$SK_1(\mathbb ZD_n)=1$. So
$$Wh(D_n) \cong \mZ^{(n+ \epsilon)/2 -\delta(n)},$$
$$K_1(\mZ D_n) \cong \mZ_2^{\epsilon +1} \oplus \mZ^{(n+ \epsilon)/2 -\delta(n)}.$$

For $G=D_n \times \mZ_2$, $F[D_n \times \mZ_2] \cong (FD_n)^2$ for any coefficient field $F$ 
with $2 \neq 0$; so $r$, $q$ are doubled. Also $(D_n \times \mZ_2)^{ab} \cong (D_n)^{ab} \times \mZ_2$; and by \cite{Ma2}, 
$SK_1(\mZ[D_n \times \mZ_2])=1$. So
$$Wh(D_n \times \mZ_2) \cong \mZ^{n+ \epsilon -2\delta(n)},$$
$$K_1(\mZ [D_n\times \mZ_2]) \cong \mZ_2^{\epsilon +2} \oplus \mZ^{n+ \epsilon-2\delta(n)}.$$

\noindent This completes the computation of the lower algebraic $K$-theory of the cell stabilizers 
for the $\G_P$-action on $\mH^3$.

%%%%%%%%%%%%%%%%%%%%%%%%%%%%%%%%%%%%%%
%%%%%%%%%%%%%%%%%%%%%%%%%%%%%%%%%%%%%%

\section{Homology of $E_{\fin} \G _P$}

In order to simplify notation, we will omit the coefficients $\kz$ in the equivariant
homology theory, and will use $\G$ to denote the Coxeter group $\G _P$
associated to a finite volume geodesic polyhedron $P\subset \mH^3$.
Our goal in this section is to explain how to compute the term $H_n^{\G} (E_{\fin}\G)$.
First recall that the $\G$ action on $\mH ^3$ provides a model for $E_{\fin}$,
with fundamental domain given by the original polyhedron $P$. If the polyhedron
$P$ is non-compact, we can obtain a cocompact model for $E_{\fin} \G$ by 
equivariantly removing a suitable collection of horoballs from $\mH ^3$. A fundamental
domain for this action is a copy of the polyhedron $P$ with each ideal vertex 
truncated. According to whether $P$ is compact or not, we will use $X$ to denote 
either $\mH^3$, or $\mH^3$ with the suitable horoballs removed. We will denote
by $\hat P$ the quotient space $X/\Gamma$, a copy of $P$ with all ideal vertices
truncated.

We observe
that for this model, with respect to the obvious $\G$-CW-structure, we have a
very explicit description of cells in $X/\G=\hat P$, as well as the corresponding stabilizers. The cells in $\hat P$ are of two distinct types. The first type of cells are 
cells from the original $P$; we call these type I cells. Namely:
\begin{itemize}
\item there is one 3-cell (the interior of $\hat P$) with trivial stabilizer,
\item the 2-cells corresponding to faces of $P$, and they all have stabilizers
isomorphic to $\mZ_2$,
\item the 1-cells corresponding to elements in $E(P)$, and their stabilizers
will be finite dihedral groups, given by the special subgroup corresponding
to the two faces intersecting in the given edge,
\item the 0-cells corresponding to elements in $V(P)$, and their stabilizers
will be 2-dimensional spherical Coxeter groups, given by the special
subgroup corresponding to the three faces containing the given
vertex.
\end{itemize}
In addition to these, we have cells arising from truncating ideal vertices in $P$, which 
we call type II cells. They are as follows:
\begin{itemize}
\item each truncated ideal vertex from $P$ gives rise to a 2-cell in $\hat P$, with
trivial stabilizer,
\item each face in $P$ incident to an ideal vertex gives rise to a 1-cell in $\hat P$
with stabilizer $\mZ_2$,
\item each edge in $P$ incident to an ideal vertex gives rise to a 0-cell in $\hat P$
with stabilizer a dihedral group (isomorphic to the stabilizer of the
edge that is being truncated). From the fact that the ideal vertex stabilizers are
2-dimensional Euclidean reflection groups, the stabilizer can only be
isomorphic to one of the groups $D_2, D_3, D_4$, or $D_6$.
\end{itemize}

Now to compute the homology group $H_n^{\G} (X)$, we recall that
Quinn has established \cite[App. 2]{Qu} the existence of an Atiyah-Hirzebruch
type spectral sequence which
converges to this homology group, with $E^2$-terms given by:
$$E^2_{p,q}=H_p(\hat P \; ;\{Wh_q(\G_{\sigma})\})
\Longrightarrow H_{p+q}^{\G}(X).$$ The complex that
gives the homology of $\hat P$ with local coefficients
$\{Wh_q( \G_{\sigma})\}$ has the form
\[
\cdots \rightarrow \bigoplus_{{\sigma}^{p+1}}^{}Wh_{q}(
\G_{{\sigma}^{p+1}}) \rightarrow \bigoplus_{{\sigma}^p}^{}Wh_q(
\G_{{\sigma}^p}) \rightarrow \bigoplus_{{\sigma}^{p-1}}^{}Wh_q(
\G_{{\sigma}^{p-1}}) \cdots \rightarrow
\bigoplus_{{\sigma}^0}^{}Wh_q( \G_{{\sigma}^0}),
\]
where ${\sigma}^p$ denotes the cells in dimension $p$, and the sum
is over all $p$-dimensional cells in $\hat P$. The $p^{th}$
homology group of this complex will give us the entries for the
$E^2_{p,q}$-term of the spectral sequence. Let us recall that
\[
Wh_q(F)=
\begin{cases}
Wh(F), & q=1 \\
\tilde {K}_0(\mathbb Z F), & q=0 \\
K_q(\mathbb Z F), & q \leq -1.
\end{cases}
\]

Note that, from the description of the stabilizers given above,
we know that there is only one 3-cell, with trivial stabilizer, and
that all the 2-cells have stabilizers that are either trivial or isomorphic to
$\mZ_2$.  But it is well
known that the lower algebraic $K$-theory of both the trivial group
and $\mZ_2$ vanishes.  In particular, for the groups of 
interest to us, we have that $E^2_{p,q}=0$ except possibly for
$p=0,1$.  It is also a well-known result of Carter \cite{C} that for
a finite group $G$, $K_n(\mZ G)=0$ for $n< -1$. This tells us that
the only possible non-zero values for $E^2_{p,q}$ occur when $p=0,1$
and $-1\leq q \leq 1$, and are given by the homology of:
\begin{equation}
0\rightarrow \bigoplus _{e\in E(P)} Wh_q(\G _e) \rightarrow \bigoplus _{v\in V(P)}
Wh_q(\G _v) \rightarrow 0
\end{equation}
So in order to finish our computation of the $E^2$-terms, we merely
need to find the various $Wh_q(\G _e)$ and $Wh_q(\G _v)$, and to
analyze the morphism appearing above.

Recall that the edge stabilizers are given by dihedral groups
$D_k$ (1-cells of type I), or are isomorphic to $\mZ_2$ (1-cells of type II). 
Note that we have already largely computed the lower
algebraic $K$-theory of dihedral groups (see Section 3). Concerning
the vertex stabilizers, we note that these will be spherical
triangle groups.  The classification of these groups is well known:
up to isomorphism, they are either the
generic $D_k \times \mZ_2$ ($k\geq 2$), or one of the three
exceptional cases $S_4$, $S_4\times \mZ_2$, and $A_5\times \mZ_2$.
We observe that, for the three exceptional cases, the lower
algebraic $K$-theory has already been computed: we refer the reader
to \cite{LO2} for $S_4$, to \cite[Section 5]{Or} for $S_4\times \mZ
/2$, and to \cite[Section 5.4]{LO2} and Section 3.2 for the 
group $A_5\times \mZ_2$.
On the other hand, for the generic case, we have already given
explicit computations for the lower algebraic $K$-theory (see
Section 3).

\subsection{Analysis of the chain complex}  Now that we know the groups
appearing in the chain complex (6), let us proceed to explain how
one can compute the $E^2$-terms for the Quinn spectral sequence for
$E_{\fin}\G$.

Recall that the only edges with potentially non-trivial $K$-groups are 
the edges of type I, with stabilizers $\G _e$ isomorphic to dihedral groups. 
Each vertex in $\hat P$ has three incident edges. Vertices of type I have 
stabilizers $\G _v$ which are spherical triangle groups and
the inclusions $\G _e\hookrightarrow \G_v$ always corresponds to the
inclusion of a special subgroup $\G_e$ into the finite Coxeter group
$\G_v$. In contrast, vertices of type II have stabilizers $\G_v$ which are
dihedral; two incident edges are of type II with stabilizer isomorphic to 
$\mZ_2$. The third incident edge is of type I, with stabilizer $G_e$ one of
the dihedral groups $D_2, D_3, D_4$ or $D_6$, and with the inclusion
$G_e\hookrightarrow G_v$ an isomorphism.

We now proceed to a case by case analysis based on the
order of the edge stabilizers arising in the truncated polyhedron $\hat P$.

\vskip 10pt

\noindent {\bf Case 1: $\mathbf{n\geq 7}$.} \hskip 10pt If we have
an edge $e\in E(\hat P)$ with stabilizer $D_n$, $n\geq 7$, then both
vertices $v,w$ appearing as endpoints of $e$ must be of type I,
with stabilizer
isomorphic to $D_n\times \mZ_2$.  Indeed, such $D_n$ do not appear as
subgroups of any other spherical triangle group, nor do they appear
as stabilizers of type II vertices.  In this case, we
observe that all the remaining edges incident to either $v$ or $w$
have to have stabilizers isomorphic to $D_2$, which we know has
vanishing lower algebraic $K$-theory.  This implies that for such an
edge $e\in E(\hat P)$, we can split off the portion of the chain complex
(6) corresponding to $e$:
$$0\rightarrow Wh_q(D_n) \rightarrow 2\cdot Wh_q(D_n\times \mZ_2) \rightarrow 0.$$
Furthermore, since $D_n\hookrightarrow D_n\times \mZ_2$ is a retract, we see
that the map above is injective, hence the homology will be concentrated in
dimension zero, and will contribute a
summand $2\cdot Wh_q(D_n\times \mZ_2) / Wh_q(D_n)$ to the corresponding
$E^2_{0,q}$.

\vskip 10pt

\noindent {\bf Case 2: $\mathbf{n =6}$.} \hskip 10pt If we have an edge
$e\in E(\hat P)$ with stabilizer $D_6$, the situation is a bit more complicated.
The endpoints $v,w$ of the edge $e$ are either of type I (with 
vertex stabilizer $D_6\times \mZ_2$) or of type II (with vertex 
stabilizer $D_6$). In both cases, the remaining edges incident to the vertices
$v,w$ have stabilizers isomorphic to $\mZ_2$ or $D_2$, which we know have
vanishing lower algebraic $K$-theory. So again, for each such edge 
$e\in E(\hat P)$, we can split off the portion of the chain complex (6)
corresponding to $e$:
$$0\rightarrow Wh_q(D_6) \rightarrow Wh_q(\G _v) \oplus Wh_q(\G _w) \rightarrow 0.$$
We now consider each of the cases $q=1,0, -1$.

For $q=1$, we have that 
$Wh(D_6)$ and $Wh(D_6 \times \mZ_2)$ both vanish, so that the
sequence above degenerates to the identically zero sequence.
In particular, the edges with stabilizer $D_6$ do not contribute to 
$E^2_{1,1}$ or $E^2_{0,1}$.

For $q=0$, we recall that $Wh_0(D_6)=0$, while $Wh_0(D_6\times \mZ_2)
\cong (\mZ_2)^2$ (see \cite[Section 5.1]{LO2}. Hence each edge with stabilizer $D_6$ makes no contribution
to $E^2_{1,0}$, while the contribution to $E^2_{0,0}$ is either $0$, 
$(\mZ_2)^2$, or $(\mZ_2)^4$ according to whether none, one,
or both of its vertices have stabilizer $D_6 \times \mZ_2$.

Finally, for $q=-1$, we have that $Wh_{-1}(D_6) \cong \mZ$ and $Wh_{-1}
(D_6\times \mZ_2) \cong \mZ ^3$. Since the natural inclusion $D_6 \hookrightarrow
D_6\times \mZ_2$ is a retract, the corresponding induced map on $Wh_{-1}$
is a split injection. This implies that edges with stabilizer $D_6$ do not 
contribute to the $E^2_{1,-1}$. At the level of $E^2_{0,-1}$, we find that 
an edge with stabilizer $D_6$ contributes either a $\mZ$, $\mZ^3$, or $\mZ^5$,
according to whether none, one,
or both of its vertices have stabilizer $D_6 \times \mZ_2$.

\vskip 5pt

\noindent {\it Remark:} Let $r$ denote the number of vertices in $P$ 
with stabilizer $D_6\times \mZ_2$, and $E_6$ denotes the number of edges
with stabilizer $D_6$. Then the overall non-trivial contribution from all the edges with
stabilizer $D_6$ can be summarized as follows: 
\begin{itemize}
\item a contribution of $\mZ_2^{2r}$ to the $E^2_{0,0}$, and
\item a contribution of $\mZ ^{E_6+2r}$ to the $E^2_{0,-1}$.
\end{itemize}

\vskip 10pt

\noindent {\bf Case 3: $\mathbf{n =5}$.} \hskip 10pt If we have an
edge $e\in E(\hat P)$ with stabilizer $D_5$, then the two endpoints
$v,w$ of the edge must be of type I. However, we still have two
possibilities for the stabilizers of the two endpoints $v,w$.
Indeed, the dihedral group $D_5$ appears as a special subgroup in
two different spherical triangle groups: $D_5\times \mZ_2$, as well
as in $[3,5]\cong A_5\times \mZ_2$.  Note that for the vertices with
stabilizer $D_5\times \mZ_2$, the remaining incident edges will
have stabilizers $D_2$, which we know has vanishing lower algebraic
$K$-theory.  On the other hand, vertices with stabilizer $A_5 \times
\mZ_2$ will have two additional incident edges, one with stabilizer
$D_3$, and one with stabilizer $D_2$.  But again, we know that these
groups have vanishing lower algebraic $K$-theory.  Hence we see that
in all cases, we can split off the portion of the chain complex (6)
corresponding to $e$:
$$0\rightarrow Wh_q(D_5) \rightarrow Wh_q(\G _v)\oplus Wh_q(\G _w) \rightarrow 0.$$
Now recall that $Wh_q(D_5)$ vanishes, except for $q=1$, where
$Wh_1(D_5)\cong \mZ$. For the group $D_5 \times \mZ_2$, the 
non-vanishing lower algebraic $K$-groups consist of $Wh_1(D_5 \times
\mZ_2) \cong \mZ ^2$, and $Wh_{-1}(D_5 \times \mZ_2) \cong \mZ$. Finally,
for the group $A_5 \times \mZ_2$, all three lower $K$-groups are 
non-trivial, with $Wh_1(A_5 \times \mZ_2)\cong \mZ ^2$, $Wh_0(A_5 \times \mZ_2)
\cong \mZ_2$, and $Wh_{-1}(A_5 \times \mZ_2) \cong \mZ ^2$.

Now for $q=1$, the chain complex gives:
$$0\rightarrow \mZ \rightarrow \mZ^2 \oplus \mZ^2 \rightarrow 0$$
where the first $\mZ$ comes from $Wh_1(D_5)$, and each 
$\mZ^2$ comes from either a copy of $Wh_1(D_5\times \mZ_2)$ or 
a copy of $Wh_1(A_5\times \mZ_2)$.  Note that since
$D_5\hookrightarrow D_5\times \mZ_2$ is a retract, the induced
mapping of $\mZ \rightarrow \mZ^2$ on Whitehead groups is split
injective.  Furthermore, the authors have shown in \cite[Section
7.3]{LO2} that the map $\mZ \rightarrow \mZ^2$ on Whitehead groups
induced by the inclusion $D_5\hookrightarrow A_5\times \mZ_2$ is
likewise split injective.  Combining these two observations, we see
that {\it regardless of the vertex stabilizers}, each edge with
stabilizer $D_5$ will contribute a $\mZ ^3$ to the $E^2_{0,1}$, and
will make no contribution to $E^2_{1,1}$.

Next we consider the case $q=0$. The chain complex degenerates to:
$$0 \rightarrow Wh_0(\G _v) \oplus Wh_0(\G _w) \rightarrow 0.$$
This tells us that each edge with stabilizer $D_5$ makes 
no contribution to $E^2_{1,0}$. As for the contribution to
$E^2_{0,0}$, each such edge contributes either a $0, \mZ_2$,
or $(\mZ_2)^2$, according to whether none, one,
or both of its vertices have stabilizer $A_5 \times \mZ_2$.

Finally, we look at the case $q=-1$. Again, the chain complex degenerates to:
$$0 \rightarrow Wh_{-1}(\G _v) \oplus Wh_{-1}(\G _w) \rightarrow 0$$
giving us that edges with stabilizer $D_5$ make
no contribution to $E^2_{1,-1}$.  For the contribution to $E^2_{0,-1}$, we
see that each such edge contributes either 
a $\mZ ^2, \mZ ^3$, or $\mZ ^4$, according to whether none, one,
or both of its vertices have stabilizer $A_5 \times \mZ_2$.

\vskip 5pt

\noindent {\sl Remark:} Let $s$ denote the number of vertices in $P$ 
with stabilizer $A_5\times \mZ_2$, and $E_5$ denote the number of edges
with stabilizer $D_5$. Then the overall non-trivial contribution from all the edges with
stabilizer $D_5$ can be summarized as follows: 
\begin{itemize}
\item a contribution of $\mZ ^{3E_5}$ to the $E^2_{0,1}$, 
\item a contribution of $\mZ_2^{s}$ to the $E^2_{0,0}$, and
\item a contribution of $\mZ ^{2E_5+s}$ to the $E^2_{0,-1}$.
\end{itemize}

\vskip 10pt

\noindent {\bf Case 4: $\mathbf{n =4}$.} \hskip 10pt If we have an
edge $e\in E(\hat P)$ with stabilizer $D_4$, then we have three
possibilities for the stabilizers of the two endpoints $v,w$.
On the one hand, the vertex could be of type II, with stabilizer
isomorphic to $D_4$. Among spherical triangle groups,
$D_4$ appears as a special subgroup in only two different
groups : $D_4\times \mZ_2$, and $[3,4]\cong S_4
\times \mZ_2$. So alternatively, we could have one or both
endpoint vertices of type I, with stabilizer $D_4\times \mZ_2$
or $S_4\times \mZ_2$.

Now in all three cases, we see that the remaining
incident edges to the vertices have stabilizers isomorphic to either 
$D_2$ or $D_3$,
which have vanishing lower algebraic $K$-theory, so we can again
split off the portion of the chain complex (6) corresponding to
$e\in E(P)$.  Observing that the group $D_4$ has no lower algebraic
$K$-theory, the portion of the chain complex further degenerates
into:
$$0\rightarrow Wh_q(\G _v) \oplus Wh_q(\G _w) \rightarrow 0,$$
and hence there will be no contribution to $E^2_{1,1}$, $E^2_{1,0}$,
and $E^2_{1,-1}$.
Further observe that all three of the groups $D_4$, $D_4\times \mZ_2$ and 
$S_4\times \mZ_2$ have vanishing $Wh_1$. So no matter what the
incident vertex groups are, we see that there is also no contribution to 
$E^2_{0,1}$ from the edges with stabilizer $D_4$.

Next let us consider what happens with $Wh_0$. 
Both $D_4\times \mZ_2$ and $S_4\times \mZ_2$ have $Wh_0$ 
isomorphic to $\mZ_4$, while $D_4$ has vanishing $Wh_0$.  In particular, 
we see that each edge with stabilizer $D_4$ will contribute $0$, $\mZ_4$, or
$(\mZ_4)^2$ to $E^2_{0,0}$, according to whether the edge joins
two, one, or no ideal vertices.

The situation for $Wh_{-1}$ is likewise more
complicated, as we have $Wh_{-1}(D_4\times \mZ_2)=0$, while
$Wh_{-1}(S_4\times \mZ_2)\cong \mZ$.  Hence the edge with
stabilizer $D_4$ will contribute $0,\mZ$, or $\mZ^2$ to $E^2_{0,-1}$
according to whether it has none, one, or two of its vertices with
stabilizer $S_4\times \mZ_2$.

\vskip 5pt

\noindent {\it Remark:} Let $t$ denote the number of vertices in $P$ 
with stabilizer $S_4\times \mZ_2$, $u$ denote the number of ideal
vertices with stabilizer $[4,4]=P4m$, and $E_4$ denote the
number of edges with stabilizer $D_4$. Then the overall non-trivial contribution 
from all the edges with stabilizer $D_4$ can be summarized as follows: 
\begin{itemize}
\item a contribution of $\mZ_4^{2E_4 - 2u}$ to the $E^2_{0,0}$, and
\item a contribution of $\mZ ^{t}$ to the $E^2_{0,-1}$.
\end{itemize}

\vskip 10pt

\noindent {\bf Case 5: $\mathbf{n \leq 3}$.} \hskip 10pt 
For edges $e \in E(\hat P)$
with stabilizer $D_3$ or $D_2$, the contribution to the
$E^2$-terms in the Quinn spectral sequence is concentrated on those
vertices with stabilizer $D_3\times \mZ_2$ or $D_2\times \mZ_2$.  Indeed, we 
have on the one hand that the lower algebraic $K$-theory of the edge groups
$D_3$ and $D_2$ vanish, so the contribution to the $E^2$-terms will come 
solely from the corresponding vertex groups.  The contribution from the vertices
having an incident edge with stabilizer $D_n$, $n\geq 4$, has already
been accounted for (in the appropriate case above). So we are left with 
dealing with vertices, all of whose incident edges are either $D_3$ or $D_2$.
The only such vertices have stabilizer $S_4$, $D_3\times \mZ_2$, or $D_2\times \mZ_2$.
Amongst these, the only non-vanishing $K$-theory appears for $D_3\times \mZ_2 \cong
D_6$ (with $K_{-1}$ isomorphic to $\mZ$) and $D_2\times \mZ_2$ (with $\tilde K_0$ 
isomorphic to $\mZ_2$).

\vskip 5pt

\noindent {\it Remark:} Let $v$ denote the number of vertices in $P$ 
with stabilizer $D_3\times \mZ_2$, and $w$ denote the number of
vertices with stabilizer $D_2\times \mZ_2$
Then the overall non-trivial contribution 
from all the edges with stabilizer $D_3$ and $D_2$ can be summarized as follows: 
\begin{itemize}
\item a contribution of $\mZ_2^{w}$ to the $E^2_{0,0}$, and
\item a contribution of $\mZ ^{v}$ to the $E^2_{0,-1}$.
\end{itemize}

\subsection{Collapsing of the spectral sequence and applications.}  Now collecting
the information from the previous few sections, we immediately see
that the $E^2$-terms in the Quinn spectral sequence all vanish, with
the possible exception of $E^2_{0,1}$, $E^2_{0,0}$, and
$E^2_{0,-1}$ (within the range $q\leq 1$).  

In particular, the spectral sequence 
always collapses at the $E^2$-stage, and yields the desired homology group.
Furthermore, from the analysis in the previous section, we obtain (see the 
Remarks after Cases 2, 3, 4, and 5) the following explicit formulas:
$$H_1^{\G}(X ;\kz) \cong \mZ^{3E_5} \oplus Q_1$$
$$H_0^{\G}(X ;\kz) \cong (\mZ_2)^{2r+s+w} \oplus (\mZ_4)^{2E_4-2u} \oplus Q_0$$
$$H_{-1}^{\G}(X ;\kz) \cong \mZ^{2r+s+t+v+2E_5+E_6} \oplus Q_{-1}$$
where in the expression above we have that:
\begin{itemize}
\item $r$ is the number of special subgroups isomorphic to $D_6\times \mZ_2$,
\item $s$ is the number of special subgroups isomorphic to $A_5\times \mZ_2$,
\item $t$ is the number of special subgroups isomorphic to $S_4\times \mZ_2$,
\item $u$ is the number of ideal vertices in $P$ with stabilizer $P4m$,
\item $v$ is the number of special subgroups isomorphic to $D_3\times \mZ_2$,
\item $w$ is the number of special subgroups isomorphic to $D_2\times \mZ_2$,
\item $E_4, E_5,$ and $E_6$ are the number of edges in $P$ with stabilizer $D_4, 
D_5$, and $D_6$ respectively.
\end{itemize}
and the terms $Q_q$ are given by:
$$Q_q \cong \bigoplus _{e\in E_l(P)} \frac{2\cdot Wh_q(\G _e \times \mZ_2)}
{Wh_q(\G _e)}$$ where $E_l(P)$ denotes the subset of edges of $P$
having ``large'' stabilizer, i.e. satisfying $\G_e=D_n$ with $n\geq
7$.

\vskip 10pt

Now let us discuss some applications of these spectral sequence computations. 
Recall that the Farrell-Jones isomorphism conjecture holds for the groups $\G _P$,
and hence the lower algebraic $K$-theory $Wh_*(\G_P)$ of $\G _P$ can be identified 
with $H_*^{\G}(E_{\vc}\G _P ;\kz)$. Furthermore, the term $H_*^{\G}(E_{\fin}\G _P ;\kz)$
computed above is a direct summand inside $H_*^{\G}(E_{\vc}\G _P ;\kz)$, and hence
a direct summand inside $Wh_*(\G _P)$ (see equation (2) in Section 2).

For $*=0, 1$, the remaining terms in equation (2) are known to be {\it purely torsion}, and
in particular, vanish when we tensor with $\mQ$. Specializing to $*=1$, and
keeping the notation from above, we immediately obtain that
$$Wh(\G _P) \otimes \mQ = \mQ ^{3E_5} \oplus (Q_1 \otimes \mQ). $$
Along with the computations in Section 3.4, this allows us to explicitly determine 
the rationalized Whitehead group:

\begin{Thm}
Let $\G _P$ be a hyperbolic reflection group with associated finite volume 
geodesic polyhedron $P\subset \mH ^3$. Then the rationalized Whitehead
group has rank:
\begin{equation}
rk \big(Wh(\G _P) \otimes \mQ \big) = \frac{3}{2} \sum_{n} E_n\big[ n+\epsilon(n) - 2 \delta (n) \big]
\end{equation}
where $E_n$ is the number of edges in $P$ with stabilizer $D_n$, 
$\epsilon(n)$ equals 1 or 2 according to whether $n$ is odd or even,
and $\delta (n)$ is the number of divisors of $n$.
\end{Thm}

\begin{proof}
By the discussion above, we need to analyze the term $Q_1 \otimes \mQ$. For a given edge with stabilizer 
$D_n$ ($n\geq 7$), we see a contribution of $2rk(Wh(D_n \times \mZ_2))-rk(Wh(D_n))$ to 
the overall rank of $Q_1 \otimes \mQ$. Appealing to the ranks of $Wh$ computed in section 3.4, we see that such an edge contributes
$$2\big( n+ \epsilon -2\delta(n)  \big) - \big((n+ \epsilon)/2 -\delta(n)\big) =
(3/2) \cdot \big( n + \epsilon - 2\delta (n) \big)$$
to the rank of $Q_1 \otimes \mQ$. Summing over all edges with 
stabilizer $D_n$, $n\geq 7$, and adding in the contribution from the edges
with stabilizer $D_5$, we obtain that:
$$rk(Wh(\G _P) \otimes \mQ) = 3E_5 + \frac{3}{2} \sum_{n \geq 7} E_n\big[ n+\epsilon(n) - 2 \delta (n) \big].$$
To conclude, we merely observe that for $n=2,3,4,6$, the expression 
$n+\epsilon(n) - 2 \delta (n)$ equals zero, while for $n=5$, we have
$5 + \epsilon(5) - 2 \delta (5) = 5 + 1 - 2(2) = 2$. So we see that the expression
computed above for $rk(Wh(\G _P) \otimes \mQ)$ is in fact equal to the expression 
appearing in equation (7), concluding the proof.
\end{proof}

\vskip 5pt

Next, let us consider the case $*=-1$. In this case, it is known that the remaining
terms in the splitting given in equation (2) {\it all vanish}. In particular, this gives us
isomorphisms
$$K_{-1}(\mZ \G _P) \cong H_{-1}^{\G}(X ;\kz) \cong \mZ^{2r+s+t+2E_5+E_6} \oplus Q_{-1}.$$
Furthermore, we have explicit computations (see Theorem 1) for the various
$K$-groups appearing in the description of $Q_{-1}$. Substituting in those 
calculations, we immediately obtain:

\begin{Thm}
Let $\G _P$ be  a hyperbolic reflection group with associated finite volume 
geodesic polyhedron $P\subset \mH ^3$. Then the group $K_{-1}(\mZ \G _P)$
is torsion-free, with rank given by the expression:
$$ 2r + s + t +v+ 2E_5 + E_6 + \sum _{n\geq 7} E_n\big(1- 3\delta(n) + 3 \tau(n) + 2\sigma_2(n) \big) $$
where $r,s,t,v$ are the number of vertex stabilizers isomorphic to $D_6\times \mZ_2$,
$A_5\times \mZ_2$, $S_4\times \mZ_2$, and $D_3\times \mZ_2$ respectively, the $E_k$ are the number
of edges in $P$ with stabilizer $D_k$, and the number theoretic quantities 
$\delta(n), \tau(n), \sigma_2(n)$ are as defined in Section 3.1.
\end{Thm}

\vskip 5pt

Finally, let us make a few comments on the case $*=0$. In this situation, we cannot deduce
any similar nice formulas for the $\tK_0(\mZ \G_P)$, the difficulties being twofold. On the one hand,
the computation of $H_0^{\G}(E_{\vc}\G _P ;\kz)$ involves knowing the reduced $\tK_0$ for dihedral
groups and products of dihedral groups with $\mZ_2$. As we saw in Section 3.3, these computations
are closely related to some difficult questions in algebraic number theory, and always yield
torsion groups. On the other hand,
the remaining terms in the expression for $\tK_0(\mZ \G_P)$ (see expression (2) in Section 2) 
can sometimes be non-zero (see Section 5), and are likewise (infinitely generated) torsion groups. 
Since it is known that these remaining terms are also
purely torsion, we can only conclude that the group $\tK_0(\mZ \G_P)$ is a torsion group (which is
already known to follow from the Farrell-Jones isomorphism conjecture for $\G_P$).

\section{Cokernels of relative assembly maps for $V\in \mathcal V$}

In this section, we focus on understanding the second term appearing in the
splitting formula given in equation (2).  We recall that this term is of the form:
\begin{equation}
\bigoplus _{V\in \mV}H_n^{V}(E_{\fin}(V)\rightarrow*)
\end{equation}
where $\mV$ consists of one representative from each conjugacy class
of the infinite groups that  arise as a stabilizers of single geodesics in $\mH ^3$,
and $H_n^V(E_{\fin}(V)\rightarrow*)$ is the cokernel of the maps on
homology $H_n^{V}(E_{\fin}(V); \kz) \rightarrow H_n^{V}(*; \kz)$,
which we call the {\it relative assembly map}.

Let $\g$ be a geodesic giving rise to a summand in expression (8).
Since the stabilizer of $\g$ is assumed to be infinite, we conclude
that $\stab(\g)$ acts cocompactly on $\g$, and hence the projection
$\pi(\g)$ of $\g$ to the fundamental domain $P$ is compact.  There
are three possibilities for the projection $\pi(\g)$:
\begin{itemize}
\item either it intersects the interior of $P$,
\item it lies entirely in the $2$-skeleton of $P$, and intersects the interior
of a face,
\item it lies entirely in the 1-skeleton of $P$.
\end{itemize}
The argument given by Lafont-Ortiz in \cite[Prop. 3.5, 3.6]{LO1}
applies verbatim to show that in the first two cases, the stabilizer
of the geodesic $\g$ has to be isomorphic to one of the groups
$\mZ$, $D_\infty$, $\mZ \times \mZ_2$, or $D_\infty \times \mZ_2$.
Now for all four of these infinite groups, it is well known that the
cokernel of the relative assembly map is trivial (see \cite{Ba2},
\cite{Wd} for the first two, and \cite{Pe} for the last two).  In
particular, these groups will make {\it no contribution} to the
expression  (8).

So let us consider geodesics of the third type.  First of all, note
that two such geodesics $\g _1, \g _2$ will have $Stab(\g _1)$
conjugate to $Stab(\g _2)$ if and only if $\pi(\g _1) = \pi (\g _2)$
(as subsets of $P$).  In particular, we see that among the groups
in $\mV$, there are {\it at most finitely many} groups of this type.
Indeed, since there are exactly $|E(P)|< \infty$ edges in the
1-skeleton of $P$, we can have at most $|E(P)|$ such subgroups 
(up to conjugacy)
inside $\G _P$.  In particular, the infinite direct sum in expression (8)
really collapses down to a finite direct sum.

We now focus on identifying (1) the actual number of such subgroups,
and (2) the corresponding cokernels for the relative assembly map.
In order to complete this process, we first observe the following:
for any such group, we can consider the action on the corresponding
geodesic $\g$, obtaining a splitting
$$0 \rightarrow \fix_{\G}(\g) \rightarrow
\sta_{\G}(\g) \rightarrow \iso_{\G, \g}(\mathbb R) \rightarrow 0$$
where $\fix_{\G}(\g)$ is the subgroup of $\G$ fixing $\g$ pointwise,
while $\iso_{\G, \g}(\mathbb R)$ is the induced action of the
stabilizer $\sta_{\G}(\g)$ on $\mR$ (identified with the geodesic
$\g$).  Note that since $\sta_{\G}(\g)$ is known to act discretely
on $\mH^3$, and cocompactly on $\g$, we immediately obtain that
$\iso_{\G, \g}(\mathbb R)$ is a discrete, cocompact subgroup of
$\iso(\mR)$, i.e. has to be isomorphic to $\mZ$ or $D_\infty$.  On
the other hand, the term $\fix_{\G}(\g)$ corresponds to the subset
fixing $\g$ pointwise, and taking a point in $\gamma$ that projects
to the {\it interior} of an edge $e$, we immediately see that this
group must be isomorphic to a dihedral group $D_n$ (coinciding with
the stabilizer of the edge $e$).  As in the previous section, let us
proceed with a case by case analysis, according to the order of the
group $\fix_{\G}(\g)$.

\vskip 10pt

\noindent {\bf Case 1: $\mathbf{n\geq 6}$.} \hskip 10pt  If we have
a geodesic $\g$ with infinite stabilizer, satisfying $\fix_{\G}(\g)\cong D_n$ with $n\geq 6$,
then we observe that $\pi(\g) \subset P$
coincides with a single edge in the 1-skeleton of $P$ (see \cite[Section 4]{LO2}), with
stabilizer $D_n$.  Furthermore, both vertex endpoints of the edge
must be non-ideal, with vertex stabilizer isomorphic to $D_n\times \mZ_2$. 

It is easy to see that the vertex stabilizers actually leave the geodesic
$\g$ invariant.
So applying Bass-Serre theory, we see that to each edge of $P$ with internal angle
$\pi/n$ with $n\geq 6$, one has an element in $\mV$ isomorphic to
$(D_n\times \mZ_2)*_{D_n}(D_n\times \mZ_2) \cong D_n\times
D_\infty$.  We note that for $V$ of the form $D_n\times D_\infty$,
the cokernel of the relative assembly map satisfies:
$$H_*^{V}(E_{\fin}(V)\rightarrow*)\cong NK_*(\mZ D_n)$$
where $NK_*(\mZ D_n)$ is the Bass Nil-group associated to the
dihedral group $D_n$ (see \cite{D}, \cite{DKR}, \cite{DQR}).  
In particular we see that {\it each geodesic extending an edge with
stabilizer $D_n$, $n\geq 6$, joining non-ideal vertices (in the case $n=6$), 
will contribute a single copy of the
Bass Nil-group for $D_n$.}

\vskip 10pt

\noindent {\bf Case 2: $\mathbf{n= 5}$.} \hskip 10pt  If we have a
geodesic $\g$ with the property that $\fix_{\G}(\g)\cong D_5$, then
we observe that, once again, the projection $\pi(\g)$ of the
geodesic into the polyhedron $P$ will consist of a single edge with
stabilizer $D_5$ (see \cite[Section 4]{LO2}).  Note that the endpoints of this 
edge must have stabilizer either $D_5\times \mZ_2$, or $A_5 \times \mZ_2$.  

Again, this allows us to use Bass-Serre theory to identify the
stabilizer of $\g$.  It will be an amalgamation of two finite groups
over the common (index two) subgroup $D_5$.  Furthermore, the
vertex groups correspond precisely to the subgroups of the 
vertex stabilizer that {\it also leaves $\g$ invariant}. It is easy to
check that, regardless of whether the vertex stabilizer is 
$D_5\times \mZ_2$ or $A_5\times \mZ_2$, this
subgroup has to be isomorphic to $D_{10}\cong D_5\times \mZ_2$.  We conclude
that each geodesic with $\fix_{\G}(\g)\cong D_5$ must have
stabilizer isomorphic to $D_{10}*_{D_5}D_{10}\cong D_5\times
D_\infty$. The cokernel of this relative assembly map 
is known to be isomorphic to the Bass Nil-group $NK_*(\mZ D_5)$ 
(\cite{LO3}, see also \cite{DKR}), which are known to vanish for 
$*\leq 1$. We conclude that
{\sl geodesics extending edges with stabilizer $D_5$ make no contribution to the lower algebraic
$K$-theory.}

\vskip 10pt

\noindent {\bf Case 3: $\mathbf{n= 4}$.} \hskip 10pt  If we have a
geodesic $\g$ with the property that $\fix_{\G}(\g)\cong D_4$, then
we observe that, once again, the projection $\pi(\g)$ of the
geodesic into the polyhedron $P$ will consist of a single edge with
stabilizer $D_4$.  In this case, the two endpoints of this edge must
have stabilizer either isomorphic to $D_4\times \mZ_2$ or to
$S_4\times \mZ_2$.  In both cases, one can see that the subgroup
of the vertex stabilizers that also leaves the geodesic invariant are
isomorphic to $D_4 \times \mZ_2$. Hence, we obtain that the stabilizer 
of $\g$ is an
amalgamation $(D_4\times \mZ_2)*_{D_4}(D_4\times \mZ_2) \cong
D_4\times D_\infty$.  We have discussed this cokernel in
\cite[Section 6.4]{LO2}: it can be identified with the Bass Nil
groups $NK_*(\mZ D_4)$ (see also \cite{D}, \cite{DKR}, \cite{DQR}).  In particular, we 
see that {\it each geodesic extending an edge with stabilizer $D_4$, and with
infinite stabilizer, will contribute 
a single copy of the Bass Nil group for $D_4$.} 

\vskip 5pt

\noindent {\it Remark:} We note that these Nil-groups
have been partially computed by Weibel \cite{We}, who showed that
$NK_0(\mZ D_4)$ is isomorphic to 
the direct sum of a  countably infinite free
$\mathbb Z/2$-module with a countably infinite free $\mathbb Z/4$-module. He 
also showed that $NK_1(\mZ D_4)$ is a countably infinite torsion group of exponent 2 or 4.

\vskip 10pt

\noindent {\bf Case 4: $\mathbf{n= 3}$.} \hskip 10pt  Geodesics $\g$
with the property that $\fix_{\G}(\g)\cong D_3$ are somewhat more
difficult to track.  The reason for this is that an edge in $P$ with
stabilizer $D_3$ can have four possible stabilizers for the
endpoints.  Indeed, the spherical triangle groups containing $D_3$
as a special subgroup include $D_3\times \mZ_2$, $S_4$, $S_4\times
\mZ_2$, and $A_5\times \mZ_2$.

Now if the projection $\pi(\g)$ of the geodesic is a union of edges
forming an interval, then we can use Bass-Serre theory to write out
the stabilizer of the geodesic. From the tessellations associated to
the four possible vertex stabilizers, we can readily see that the
geodesic is reflected whenever the endpoint has stabilizer $D_3\times
\mZ_2$, $S_4\times \mZ_2$, or $A_5\times \mZ_2$.  In all three
cases, one sees that the subgroup of the vertex stabilizer that
leaves the $\g$ invariant is in fact isomorphic to $D_3\times
\mZ_2$.  We conclude that in this case, the stabilizer of $\g$ is
isomorphic to $(D_3\times \mZ_2) *_{D_3}(D_3\times \mZ_2) \cong
D_3\times D_\infty$.  But the authors have shown that the cokernel
of the relative assembly map for this group is isomorphic to the
Bass Nil-group $NK_*(\mZ D_3)$ (see \cite[Section 5]{LO1},
as well as \cite{D}, \cite{DKR}, \cite{DQR}).

Alternatively, the projection $\pi(\g)$ of the geodesic could be a
union of edges forming a closed loop in the 1-skeleton $P^{(1)}$ 
of $P$.  In this case, we see that the stabilizer of $\g$ fits into a short
exact sequence:
$$0 \rightarrow D_3 \rightarrow \sta_{\G}(\g) \rightarrow \mZ
\rightarrow 0$$ and hence can be written as a semidirect product
$D_3 \rtimes_\alpha \mZ$. In this case, the cokernel of the relative
assembly map will be a Farrell Nil-group $NK_i(\mZ D_3, \alpha)$.

Summarizing this discussion, we see that each orbit of a periodic
geodesic in $\mH ^3$ which is pointwise fixed by a $D_3$ contributes
a single copy of a Farrell Nil-group $NK_*(\mZ D_3; \alpha)$ (for a suitable
automorphism $\alpha \in \aut(D_3)$). Finally, we remark that for $*=0,1$, the 
Farrell Nil-groups $NK_*(\mZ D_3; \alpha)$ are known to vanish, irrespective of the 
automorphism $\alpha$. Hence we obtain that
{\sl the geodesics extending edges with stabilizer $D_3$ make 
no contribution to the lower algebraic
$K$-theory.}

\vskip 10pt

\noindent {\bf Case 5: $\mathbf{n= 2}$.} \hskip 10pt  Geodesics $\g$
with $\fix_{\G}(\g)\cong D_2$ are the most difficult ones to handle.  
The primary difficulty is that every spherical triangle
group contains $D_2$ as a special subgroup. Now assume we
have such a geodesic $\g$, and consider its projection into the 
1-skeleton $P^{(1)}$ of $P$. The projection is either:
\begin{itemize}
\item a union of edges forming a closed loop inside $P^{(1)}$, or
\item a union of edges forming a path inside $P^{(1)}$.
\end{itemize}
If the projection is a path, Bass-Serre theory applies,
and the stabilizer of the geodesic $\g$ will have to be isomorphic
to one of the groups $D_4*_{D_2}D_4$, $D_4*_{D_2}(D_2\times \mZ_2)$,
or $(D_2\times \mZ_2)*_{D_2}(D_2\times \mZ_2) \cong D_2\times
D_\infty$ (depending on the nature of the endpoints of the path).  
In this situation, the authors have established (see
\cite[Section 4]{LO1} and \cite[Sections 6.2, 6.3]{LO2}) that for
all three of these groups, the cokernels of the relative assembly
map are isomorphic to the Bass Nil-group $NK_*(\mZ D_2)$ corresponding
to the canonical index two subgroup isomorphic to $D_2\times \mZ$.
These Bass Nil-groups are known to be
isomorphic to $\bigoplus _\infty \mZ_2$, a countable direct
sum of $\mZ_2$, in dimensions $*=0$ and $*=1$.

Alternatively, if the projection is a closed loop, then from the
short exact sequence:
$$0 \rightarrow D_2 \rightarrow \sta_{\G}(\g) \rightarrow \mZ
\rightarrow 0$$ we have that the stabilizer is of the form
$D_2\rtimes _\alpha \mZ$, $\alpha \in \aut(D_2)$.  We now claim that the 
geometry of the situation forces $\alpha =Id$, i.e. the stabilizer is in fact a direct
product $D_2 \times \mZ$. In order to see this, 
we first observe that $\aut(D_2)= \aut(\mZ_2 \times \mZ_2) \cong S_3$, 
given by an arbitrary permutation of the three non-zero elements 
in $D_2$. Let us try to rule out the various automorphisms in $\aut(D_2)$.

First, let us denote by $g,h$ the reflections in the hyperplanes  $P_1,P_2$ 
extending the two faces of the polyhedron incident to one of the edges in the closed loop. 
Now the fixed subgroup of $\g$ can be identified with the subgroup
$D_2$ consisting of $\{1, g, h, gh\}$, and the elements $g,h$ are the 
canonical generators of the special subgroup $D_2$. We also have
that $\alpha$ permutes the subset $\{g,h, gh\}$. But observe that
$g,h$ are reflections, whereas their product $gh$ is a rotation 
by $\pi$ around the geodesic $\gamma$. Since rotations are 
never conjugate to reflections, this implies that $gh$ must be
fixed by the permutation $\alpha$. So the only possibility that is 
left is where $\alpha$ interchanges $g$ and $h$.

In order to rule out this last possibility, we can look at the
element $\tau$ in the stabilizer of $\gamma$ that acts via a minimal translation along
$\gamma$. Note that
$\tau$ either maps each $P_i$ to itself, or interchanges $P_1$ and $P_2$. 
So to rule out the case where $\alpha$ interchanges the two reflections
$g$ and $h$, it is sufficient to identify the element $\tau$, and verify that
it leaves invariant each of the hyperplanes $P_i$.

We now focus on explicitly describing the element $\tau$ in the group $\G$.
Given a pair of consecutive edges in this loop, we have that the corresponding
common vertex of intersection must have stabilizer $\G_v$ of the form $D_k\times \mZ _2$, 
with $k$ an {\it odd} integer. The
two incoming edges with stabilizer $D_2$ correspond to the special subgroups $\G _{e_i}$,
$G_{e_j}$ of 
$D_k \times \mZ _2$ of the form $\langle g \rangle \times \mZ _2$, where $g$ is one
of the two canonical reflections generating $D_k$. In this situation, there is a unique
element $\nu(v, e_i)$ of the group $\G _v$ whose action takes the geodesic extending
the edge $e_i$ to the geodesic extending the edge $e_j$.
Of course, we have the obvious relation $\nu(v, e_j) = \nu(v, e_i)^{-1}$. In concrete
terms, the element $\nu(v, e_i)$ can be described as follows: it is simply the longest 
word in the group $\G_V$ (with respect to the Coxeter generating set). Note that 
the element $\nu (v, e_i)$ is always a rotation inside $\iso(\mH^3)$.
Geometrically, this rotation of $\mH^3$ fixes the vertex $v$, and at the level of the 
spherical tessellation of the unit tangent sphere at $v$, takes the spherical triangle 
corresponding to the polyhedron $P$ to the spherical triangle which is directly opposite. 

Now assume that the loop of interest is given cyclicly by the sequence of 
edges and vertices
$\{e_1, v_1, e_2, v_2, \ldots ,e_n, v_n\}$.
We can consider each of the elements 
$\nu(v_i, e_i)$, and observe that the product 
$$\nu(v_1,e_1)\cdot \nu(v_2, e_2)\cdot \ldots \cdot \nu(v_n,e_n) \in \G$$ clearly 
stabilizes the geodesic $\gamma$ extending the edge $e_1$. Furthermore, 
this element acts via a minimal translation along $\gamma$, and hence
can be taken as an explicit description of the desired element $\tau$. Note that 
Deodhar \cite{De} considered similar elements in the general setting of Coxeter
groups (see also Davis' book \cite[Section 4.10]{Da}). We now are left with
verifying that this element $\tau$ leaves each of the two hyperplanes $P_1, P_2$ invariant.

To check this last statement, we first consider how the element $\nu(v,e_j)$ acts on the
hyperplanes whose intersection defines the vertex $v$. We note that there are three such
hyperplanes $P_1, P_2, P_3$, labelled so that $P_1\cap P_2 = e_j$ and $P_2\cap P_3=e_i$.
In particular, we have that the hyperplanes $P_1$ and $P_3$ intersect at an angle $\pi/k$,
with $k$ {\it odd}. Now from the explicit formula for $\nu(v, e_j)$, it is immediate that it
leaves $P_2$ invariant, and interchanges $P_1$ and $P_3$. In other words, the element
$\nu(v,e_j)$ interchanges the two hyperplanes whose intersection is the edge with internal
dihedral angle $\pi/k$. We would now like to use this to compute the effect of the element
$\tau$ on the original pair of hyperplanes.

We note that the loop of interest is a simple loop in the 1-skeleton $P^{(1)}$ of the polyhedron
$P$, hence can be thought of as a simple closed curve on $\partial P \cong S^2$. In particular,
this loop separates $\partial P$ into precisely two connected components $U_1, U_2$, and 
without loss of generality, we have that the faces 
$F_1,F_2$ whose intersection forms the edge $e_1$ satisfy $F_i\subset U_i$. Considering 
vertex $v_1$, let us think of the action of $\nu(v_1,e_1)$ on the two hyperplanes $P_1,P_2$. 
Since $v_1$ has degree three, we have an edge $f_1$ in $P^{(1)}$ incident to the loop, and this edge
must have internal angle of the form $\pi/k$ ($k$ odd). Observe that $f_1$ is contained in one
of the two components $U_1,U_2$. From the discussion in the previous paragraph, the effect 
of $\nu(v_1,e_1)$ is to interchange the hyperplanes extending the faces adjacent to $f_1$, and
to leave invariant the hyperplane extending the face opposite $f_1$. But observe that 
the two faces incident to $f_1$ are contained in the same component $U_i$, while the
opposite face to $f_1$ is contained in the other component. This forces the action of
$\nu(v_1,e_1)$ to respect the components $U_1$, $U_2$. Similarly, we see that each 
of the elements $\nu(v_i, e_i)$ respect the individual components, which forces their 
product $\tau$ to similarly respect the components. Since $\tau$ maps the hyperplane
$P_1$ extending $F_1$ to the hyperplane extending a face which:
\begin{itemize}
\item is incident to $e_1$, i.e. is either $F_1$ or $F_2$, and
\item is in the same connected component $U_1$ as $F_1$
\end{itemize}
we conclude that $\tau$ leaves $P_1$ (and likewise $P_2$) invariant. This forces $\alpha = Id$, 
ensuring that the stabilizer of the corresponding geodesic
must be isomorphic to the direct product $D_2\times \mZ$. For this 
group, the cokernel of the relative assembly map is the classic Bass Nil-group
$NK_*(\mZ D_2)$.

\vskip 10pt

Finally, let us comment on the number of copies of this Bass Nil-group that 
will appear in our computation. This requires counting $\G _P$ orbits of geodesics
whose stabilizer is infinite, and is fixed by a subgroup isomorphic to $D_2$. But this
is actually not too difficult. Indeed, such a geodesic
has to project to the subset of the 1-skeleton of $P$ consisting of
edges with internal dihedral angle $= \pi/2$. So given the polyhedron $P$,
restrict to this subset of the 1-skeleton $P^{(1)}$, obtaining a graph $\mG
_2$. Since the 1-skeleton of $P$ has the property that every vertex
has degree $\leq 4$, the subgraph $\mG_2$ inherits this same
property.  Now disconnect this graph along all vertices of degree
$3$ or $4$, resulting in a collection of intervals and loops. Finally,
disconnect the graph at any vertex of degree $2$, having the property
that the {\it third incident edge in $\mG$} has internal dihedral angle
which is {\it even}. 
Discard all intervals with the property that one of their endpoints
came from a vertex of degree $4$.  Then there is a bijective correspondence 
between:
\begin{enumerate}
\item connected components of the resulting graph $\hat \mG _2$,
\item $\G_p$-orbits of geodesics $\g \subset \mH^3$ with infinite
stabilizer and $\fix_{\G}(\g)\cong D_2$.
\end{enumerate}
By the discussion in the last couple of pages, we conclude that 
{\sl geodesics extending the edges with stabilizer $D_2$ contribute a total of 
$|\pi_0(\hat \mG _2)|\cdot NK_*(\mZ D_2)$
to the lower algebraic $K$-theory.}

\vskip 10pt

\noindent {\it Remark:} We note that for $i=0,1$ each of these Bass Nil-groups 
$NK_i(\mZ D_2)$ is isomorphic
to $\bigoplus _{\infty}\mZ _2$, the direct sum of countably many copies of $\mZ_2$
(see \cite[Lemma 5.3, 5.4]{LO1}).

%%%%%%%%%%%%%%%%%%%%%%%%%%%%%%%%%%%%%%

\section{Appendix: concrete examples.}

To illustrate the methods discussed in this paper, we now proceed to work
through the lower algebraic $K$-theory for some concrete examples. Let us 
start with a relatively simple class of examples. Consider the
groups $\Lambda _n$, $n\geq 5$, given by the following presentation:

$$\Lambda _n:= \Bigg\langle
y, z, x_i, \hskip 5pt 1\leq i \leq n
\hskip 5pt \Bigg|  \hskip 5pt
\parbox{2.5in}
{\centerline{$y^2, z^2,$}

\centerline{$x_i^2,  (x_ix_{i+1})^2, (x_iz)^3, (x_iy)^3$, \hskip 5pt $1\leq i \leq n $}}
\Bigg\rangle$$
The groups $\Lambda _n$ are Coxeter groups, and the presentation given
above is in fact the Coxeter presentation of the group. The corresponding Coxeter 
diagram appears in Figure 1(a).

\begin{example}
For the groups $\Lambda _n$ whose presentations are given above,
\begin{enumerate}

\item the Whitehead group is given by
$$Wh(\Lambda _n) \cong n\cdot 
NK_1(D_2) \cong \bigoplus _{\infty}\mZ _2;$$

\item the $\tilde K_0$  is given by
$$\tK _0(\mZ \Lambda _n) \cong n\cdot 
NK_0(D_2) \cong \bigoplus _{\infty}\mZ _2;$$

\item the $K_{-1}$ always vanishes.
\end{enumerate}
\end{example}

%%%%%%%%%%%%%%%%%%%%%% pictex diagram for Figure 1 %%%%%%%%%%%%%%%%%%%%%%%%%%%%%
\begin{figure}  %[htbp]
\label{graph}
\begin{center}
%Figure 1
\vbox{\hbox{\beginpicture
  \setcoordinatesystem units <1.8cm,1.8cm> point at -.4 2.5
  \setplotarea x from -1 to 5, y from -.3 to 3

% ``smallbullet''  adjust hskip and circle size as needed  
\def\smallbul{\hskip .8pt\circle*{2.2}}
%  \linethickness=1.4pt   %for  \putrule 
%test  \putrule from 0 0 to 1 0
  %\setlinear
  % Make some type plots a little thicker - like the \putrule
  \setplotsymbol ({\circle*{.5}})
  \plotsymbolspacing=.3pt        % Make a little smoother line (default .4pt).
  \plot 1 0  0 1  1 2  2 1  1 0 /
  \plot 1 0  .4 1  1 2  1.6 1  1 0 /
%  \setdashes
%\plot   .4 1  1.6 1 /     % in dots below
  \put {$\bullet$} at 1 0
  \put {$\bullet$} at 0 1
  \put {$\bullet$} at .4 1
  \put {$\bullet$} at 1.6 1
  \put {$\bullet$} at 2 1
  \put {$\bullet$} at .99 2
  \put {(a)} [t] at  1 -.3
{\setplotsymbol ({\circle*{1.2}}) %%For bolder dots
\put{$\cdots$} at 1 1 
 % \setdots
  %\plot   .4 1  1.6 1 /
  }
%% next diagram (b)
  \setsolid
  \plot 4 0   3.2 .6  3.7 1   4.3 1    4.8 .6   4 0 /
  \plot 4 1.4  3.2 2  3.7 2.4  4.3 2.4   4.8 2  4 1.4 /
%test  \putrule from 3.7 2.4 to 4.3 2.4
\setplotsymbol ({\circle*{1.2}}) %%For bolder dots
\setdots
%\setdashes
  \plot 4 0 4 1.4 /
  \plot 3.2 .6  3.2 2 /
  \plot 3.7 1  3.7 2.4 /
  \plot  4.3 1 4.3 2.4 /
  \plot  4.8 .6 4.8 2 /
  \plot   4 0 4 1.4 /
  \put {\smallbul} at 4 0
  \put {\smallbul} at 3.2 .6
  \put {\smallbul} at 3.7 1
  \put {\smallbul} at 4.3 1 
  \put {\smallbul} at 4.8 .6
  \put {\smallbul} at 4 1.4
  \put {\smallbul} at 3.2 2
  \put {\smallbul} at 3.7 2.4
  \put {\smallbul} at 4.3 2.4
  \put {\smallbul} at 4.8 2
  
  \put {(b)} [t] at  4 -.3
  \endpicture}}
\caption{ }
\end{center}
\end{figure}
%%%%%%%%%%%%%%%%%%%%%End of Figure 1 %%%%%%%%%%%%%%%%%%%%%%%%%%%%%%%%%%%%%%%%%%%

\begin{proof}
The groups $\Lambda _n$ arise as hyperbolic reflection groups, with underlying polyhedron
$P$ the product of an $n$-gon with an interval. An illustration of the polyhedron associated
to the group $\Lambda _5$ is shown in Figure 1(b), where again, ordinary edges have dihedral 
angle $\pi/3$, while dotted edges have dihedral angle $\pi/2$. In general, the polyhedron associated
to the group $\Lambda _n$ is combinatorially a product of the $n$-gon with an interval. This polyhedron
has exactly two faces which are $n$-gons, and the dihedral angle along the edges of these two faces
is $\pi/3$. All the remaining edges have dihedral angle $\pi/2$.

To begin with, we observe that for the associated polyhedron, every edge has stabilizer $D_2$ or 
$D_3$, giving us $E_k=0$ for $k\geq 4$. Furthermore, for the associated polyhedron, 
every vertex has stabilizer $S_4$, implying that $r=s=t=v=0$. Applying Theorem 6,
we immediately obtain that $K_{-1}(\mZ \Lambda _n)=0$. Applying Theorem 5, we also obtain
that $Wh(\Lambda _n) \otimes \mQ =0$. Note that the discussion in Section 4.2 actually establishes
that 
$$ H_{1}^{\Lambda _n}(E_{\fin} \Lambda _n ;\kz) =0$$
So to complete the computation of $Wh(\Lambda _n)$, we need to identify the remaining terms in
the splitting described in equation (2). From the discussion in our Section 5, we see that the next
step is to understand geodesics in $\mH^3$ whose projection under the $\Lambda _n$-action
lies in the $1$-skeleton of the polyhedron $P$. But it is easy to see that, up to the $\G _P$-action, these give:
\begin{itemize}
\item two distinct geodesics with stabilizer $D_3\times \mZ$, which project to the boundary
of the two $n$-gons appearing in $P$, and
\item $n$ distinct geodesics whose fixed subgroup is $D_2$, each of which projects to a
single edge lying between the two $n$-gons in the polyhedron $P$.
\end{itemize}
Now we know (Section 5, Case 4) that geodesics with stabilizer $D_3\times \mZ$ yield no 
contribution to the splitting in equation (2). On the other hand, each of the geodesics with
fixed subgroup $D_2$ contributes (Section 5, Case 5) a copy of $NK_1(\mZ D_2)$, which is 
isomorphic to a countable infinite direct sum of $\mZ _2$.

Finally, let us consider the case of $\tK _0$. Since we have $r=s=u=w=E_4=0$, we have
(see Section 4.2) that $ H_{0}^{\Lambda _n}(E_{\fin} \Lambda _n ;\kz) = 0$. Now the 
discussion in the previous paragraph, combined with the splitting in equation (2), gives
us that $\tilde K_0(\mZ \Lambda _n) \cong n\cdot NK_0(\mZ D_2)$. But it is known that these
Bass Nil-groups are isomorphic to the countable direct sum of infinitely many copies of $\mZ _2$,
concluding our computation.

\end{proof}

Next, let us consider a somewhat more complicated family of examples.  For
an integer $n\geq 2$, we consider the group $\G _n$, defined by the following
presentation:
$$\G _n:= \Bigg\langle x_1, \ldots , x_6   \hskip 5pt \Bigg|  \hskip 5pt
\parbox{3.3in}
{\centerline{$x_i^2, (x_1x_2)^n, (x_1x_5)^2, (x_1x_6)^2, (x_3x_4)^2, (x_2x_5)^2, (x_2x_6)^2$}

\centerline{$(x_1x_4)^3, (x_2x_3)^3, (x_4x_5)^3, (x_4x_6)^3, (x_3x_5)^3, (x_3x_6)^3$}}
\Bigg\rangle $$
Observe that the groups $\G_n$ are Coxeter groups, and that the presentation given
above is in fact the Coxeter presentation of the group. The corresponding Coxeter 
diagram appears in Figure 2(a).

\vskip 10pt

\begin{example}
For the groups $\G _n$ whose presentations are given above, 
\begin{enumerate}
\item the rationalized Whitehead group is given by $$Wh(\G _n)\otimes \mQ \cong 
\mQ ^{(3/2)\cdot (n+\epsilon(n) - 2\delta (n))}$$

\item the Whitehead group is given by
$$Wh(\G _n) \cong \mZ ^{(3/2)\cdot (n+\epsilon(n) - 2\delta (n))}\oplus (1+ 2\epsilon(n))\cdot 
NK_1(\mZ D_2) \oplus NK_1(\mZ D_n)$$

\item the $\tilde K_{0}$ is given by
$$\tilde K_{0}(\mZ \G_n) \cong 
\begin{cases} 
\frac{2\cdot \tK_0(\mZ [D_n\times \mZ_2])}{\tK_0(\mZ [D_n])} \oplus  (1+2\epsilon(n)) \cdot NK_0(\mZ D_2) \oplus NK_0(\mZ D_n)& n\geq 7 \\
\mZ_2 ^4 \oplus 5\cdot NK_0(\mZ D_2) \oplus NK_0(\mZ D_6) & n=6 \\
\mZ_4 \oplus 3\cdot NK_0(\mZ D_2) & n=5 \\
\mZ_4 \oplus 5\cdot NK_0(\mZ D_2) & n=4 \\
3\cdot NK_0(\mZ D_2) & n=3 \\
\mZ_2 \oplus 6\cdot NK_0(\mZ D_2) & n=2
\end{cases}$$

\item the $K_{-1}$ is given by
$$K_{-1}(\mZ \G_n) \cong 
\begin{cases} 
\mZ ^{1- 3\delta(n) + 3\tau(n) +2\sigma _2(n)} & n\geq 7 \\
\mZ ^5 & n=6 \\
\mZ ^2 & n=5 \\
0 & n=4 \\
\mZ ^2 & n=3 \\
0 & n=2 \\
\end{cases}$$
\end{enumerate}
\end{example}

\begin{proof}
To verify the results stated in this example, we first observe that the Coxeter
groups $\G _n$ arise as hyperbolic reflection groups, with underlying polyhedron
$P$ a combinatorial cube. The geodesic 
polyhedron associated to $\G _n$ is shown in the Figure 2(b). In the illustration,
the bold edge has internal dihedral angle $\pi/n$, the ordinary edges have 
internal dihedral angle $\pi/3$, and the dotted edges have internal 
dihedral angle $\pi/2$.

To compute the rationalized Whitehead group, we just apply our Theorem 5. The
polyhedron $P$ has five edges with stabilizer $D_2$, six edges with stabilizer $D_3$,
and one edge with stabilizer $D_n$. Evaluating equation (7) gives us that the rank of
$Wh(\G _n) \otimes \mQ$ is equal to ${(3/2)\cdot (n+\epsilon(n) - 2\delta (n))}$.

\vskip 5pt

%%%%%%%%%%%%%%%%%%%%%% pictex diagram for Figure 2 %%%%%%%%%%%%%%
\begin{figure}   %[htbp]
\label{graph}
\begin{center}
%Figure 2
\vbox{\beginpicture
  \setcoordinatesystem units <1.6cm,1.6cm> point at -.4 2.5
  \setplotarea x from -.8 to 3, y from -.6 to 2.5
  
  % ``smallbullet''  adjust hskip and size as needed
  \def\smallbul{\hskip .8pt\circle*{2.2}}
  \linethickness=.7pt
  %\putrule from 0 0 to 1 0
  %\setlinear
  % Make some type plots a little thicker - like the \putrule
%  \setplotsymbol ({\circle*{.5}})
  \plotsymbolspacing=.3pt        % Make a little smoother line (default .4pt).
  \plot .5 0  -.1 1  .5 2  1.1 1  .5 0  2 0  2 2 .5 2 /
%
%  \setdashes

  \put {$\bullet$} at .5 0
  \put {$\bullet$} at -.1 1
  \put {$\bullet$} at  .5 2
  \put {$\bullet$} at 1.1 1
  \put {$\bullet$} at .5 0
  \put {$\bullet$} at 2 0 
  \put {$\bullet$} at 2 2
  \put {$n$} [l] at  2.1 1
  \put {(a)} [t] at  1 -.3

%% next diagram (b) in figure 2 %%%%%%%%%%%%%%%
  \setsolid
  \plot 3 0  3 1.5  4.5 1.5  4.5 0  5.4 .75  4.56 .75 /
  \plot 4.44 .75   3.9 .75  3 0 /

%%Top bold line
\linethickness=2pt
  \putrule from 3.9 2.25 to 5.4 2.25 

\setdots
\setplotsymbol ({\circle*{1.2}}) %%For bolder dots
%\setdashes
  \plot 3 0  4.5 0  /
  \plot 3.9 .75   3.9 2.25  /
  \plot 5.4 .75  5.4 2.25  /
  \plot 3 1.5  3.9 2.25  /
  \plot 4.5 1.5   5.4 2.25  /
  \put {\smallbul} at 3 0 
  \put {\smallbul} at 4.5 0
  \put {\smallbul} at 3 1.5
  \put {\smallbul} at 4.5 1.5
  \put {\smallbul} at 3.9 .75
  \put {\smallbul} at 5.4 .75
  \put {\smallbul} at 3.9 2.25
  \put {\smallbul} at 5.4 2.25
  
  \put {(b)} [t] at  4 -.3
  \endpicture}

\caption{ }
\end{center}
\end{figure}
%%%%%%%%%%%%%%%%%%%%%%%%%%% End of figure 2 %%%%%%%%%%%%%%%

Now while Theorem 5 gives us a simple formula for the {\it rationalized} Whitehead
group, it only requires a little bit more work to calculate the integral Whitehead group. 
In order to do this, we exploit the splitting given in equation (2) (see Section 2.4):
$$
Wh(\G _n) \cong H_1^{\G _n}(E_{\fin}(\G _n);\mathbb K\mathbb Z^{-\infty})
\oplus \bigoplus _{V\in \mathcal V}H_1^{V}(E_{\fin}(V)\rightarrow*).
$$
For the groups $\G _n$, the first term in the splitting is computed in Section 4.2 
(see the argument for Theorem 5), and is free abelian of rank 
${(3/2)\cdot (n+\epsilon(n) - 2\delta (n))}$. As far as the second term 
in the splitting is concerned, we apply the procedure in Section 5. 
The edge with stabilizer $D_n$ contributes a single Bass Nil-group 
$NK_1(\mZ D_n)$ to the
splitting. The collection of edges with stabilizer $D_3$ form a closed
cycle, which is the image of a single geodesic in $\mH^3$. This gives
rise to a single Farrell Nil-group $NK_1(\mZ D_3, \alpha)$ (for a suitable
automorphism $\alpha \in \aut(D_3)$); but these groups are known to 
vanish. Finally, the edges with stabilizer $D_2$ correspond to either 
three or five geodesics in $\mH^3$, according to whether $n$ is odd 
or even. Overall, these contribute
$1+2\epsilon(n)$ copies of the Bass Nil-group $NK_1(\mZ D_2)$ to 
$Wh(\G _n)$. This completes our computation of $Wh(\G _n)$.

\vskip 5pt

Next, let us compute $K_{-1}(\mZ \G _n)$. We 
first observe that six of the eight vertices in $P$ have stabilizer $S_4$,
while the remaining two vertices have stabilizer $D_n\times \mZ_2$.
This tells us that, in the notation of Theorem 6, $s=t=0$. Now:
\begin{itemize}
\item if $n=2$, then we additionally have that $E_k=0$ for $k\geq 5$,
and $r=v=0$,
\item if $n=3$, then we have that $E_k=0$ for $k\geq 5$,
and $r=0$, while $v=2$,
\item if $n=4$, then we have that $E_k=0$ for $k\geq 5$,
and $r=v=0$,
\item if $n=5$, then we have $E_k=0$ for $k\geq 6$, $E_5=1$, and $r=v=0$,
\item if $n=6$, then we have $v=E_5=0$, $E_6=1$, $E_k=0$ for $k\geq 7$,
and $r=2$,
\item if $n\geq 7$, then we have $r=v=E_5=E_6=0$, and within the range
$k\geq 7$, all the $E_k$ vanish except $E_n=1$.
\end{itemize}
Applying Theorem 6 completes the computation of $K_{-1}(\mZ \G_n)$.

\vskip 5pt

Lastly, let us compute $\tK _0(\mZ \G_n)$. We again make use of the splitting 
given in equation (2). The first term in the splitting, 
$H_0^{\G _n}(E_{\fin}(\G _n);\mathbb K\mathbb Z^{-\infty})$, is given by a particularly
simple expression (see Section 4.2). Observe that, regardless of the value of 
$n$, we have that $s=u=0$; so the expression for the first term in the splitting
reduces to 
$$H_0^{\G _n}(E_{\fin}(\G _n) ;\kz) \cong (\mZ_2)^{2r + w} \oplus (\mZ_4)^{2E_4} \oplus Q_0$$
As such, we see that the first term vanishes, except in the following three cases:
\begin{itemize}
\item $n=2$, where $r=E_4=0$, and $w=1$, so the first term is $\mZ_2$
\item $n=4$, where $r=w=0$, $E_4=1$, and $Q_0\cong 0$, so the first term is $\mZ_4$,
\item $n=6$, where $r=2$, $w=E_4=0$, and $Q_0\cong 0$, and hence the first term is $\mZ_2^4$,
\item $n\geq 7$, where $r=w=E_4=0$, and hence the first term coincides with the group
$Q_0= 2\tK_0(\mZ [D_n\times \mZ _2]) / \tK_0(\mZ D_n)$.
\end{itemize}
For the second term in the splitting (2), we apply the methods from Section 5. The
second term is determined by orbits of geodesics in $\mH ^3$ whose stabilizer is infinite,
and which project to the $1$-skeleton of the geodesic polyhedron associated to the group 
$\G _n$. But these geodesics were determined earlier, when we discussed the computation
of the Whitehead group. An identical analysis gives us that the second term is just
$(1+ 2\epsilon (n)) \cdot NK_0(\mZ D_2) \oplus NK_0(\mZ D_n)$. Combining these two terms
completes the computation of the $\tK_0$, and concludes the computations for this
example.
\end{proof}

\end{document}